\documentclass[12pt]{article}
\usepackage[english]{babel}
\usepackage{amsmath,amsthm,amsfonts,amssymb,epsfig,verbatim}
\usepackage[left=1in,top=1in,right=1in]{geometry}
\usepackage[normalem]{ulem}
\usepackage{xcolor}

\newtheorem{thm}{Theorem}[section]
\newtheorem{lem}[thm]{Lemma}

\newtheorem{df}[thm]{Definition}

\newtheorem{rem}{Remark}

\numberwithin{equation}{section}

\newcommand{\E}{\mathbf{E}}

\newcommand{\prob}{\mathbf{P}}

\newcommand{\R}{\mathbb{R}}

\newcommand{\Z}{\mathbb{Z}}
\newcommand{\N}{\mathbb{N}}
\newcommand{\T}{\mathbb{T}}

\DeclareMathOperator{\diam}{diam}
\DeclareMathOperator{\Var}{Var}

\title{Universality of the time constant for $2D$ critical first-passage percolation}
\author{Michael Damron \thanks{The research of M. D. is supported by an NSF CAREER grant.} \\ \small{Georgia Tech}  \and Jack Hanson \thanks{The research of J. H. is supported by NSF grant DMS-161292.}\\ \small{City College of New York} \and Wai-Kit Lam \\ \small{University of Minnesota}}

\begin{document}
	\maketitle 
	\begin{abstract}
		We consider first-passage percolation (FPP) on the triangular lattice with vertex weights $(t_v)$ whose common distribution function $F$ satisfies $F(0)=1/2$. This is known as the critical case of FPP because large (critical) zero-weight clusters allow travel between distant points in time which is sublinear in the distance. Denoting by $T(0,\partial B(n))$ the first-passage time from $0$ to $\{x : \|x\|_\infty = n\}$, we show existence of the ``time constant'' and find its exact value to be
		\[
		\lim_{n \to \infty} \frac{T(0,\partial B(n))}{\log n} = \frac{I}{2\sqrt{3}\pi} \text{ almost surely},
		\]
		where $I = \inf\{x > 0 : F(x) > 1/2\}$ and $F$ is any critical distribution for $t_v$. This result shows that the time constant is universal and depends only on the value of $I$. Furthermore, we find the exact value of the limiting normalized variance, which is also only a function of $I$, under the optimal moment condition on $F$. The proof method also shows an analogous universality on other two-dimensional lattices, assuming the time constant exists.
	\end{abstract}
	
	\section{Introduction}
	
	Let $\T$ be the triangular lattice. We will take $\T$ to be embedded in $\R^2$ with vertex set $\Z^2$ and with edges between points of the form $(x_1,y_1)$ and $(x_2,y_2)$ with either (a) $\|(x_1,y_1)-(x_2,y_2)\|_1 = 1$ or (b) both $x_2=x_1+1$ and $y_2 = y_1 - 1$.
	
	We will consider first-passage percolation on $\T$, which is defined as follows. Let $(\omega_x)_{x\in \Z^2}$ be a family of i.i.d. uniform random variables on $(0,1)$. Fix a distribution function $F$ with $F(0^-)=0$. We define $t_x = F^{-1}(\omega_x)$ (so that $t_x$ has distribution $F$), where for $t\in (0,1)$,
	\[
	F^{-1}(t) := \inf\{y \in \R: F(y) \geq t\}. 
	\]
	A path is a sequence of vertices $(x_1, \ldots, x_n)$ with $x_i$ being adjacent to $x_{i+1}$ for all $i=1, \ldots, n-1$, and a circuit is a path $(x_1, \ldots, x_n)$ with $x_1=x_n$. We will always assume that paths and circuits are self-avoiding. (A self-avoiding circuit is one such that $(x_1, \dots, x_{n-1})$ is self-avoiding.) For a path $\gamma = (x_1, \ldots, x_n)$, we define its passage time by
	\[
	T(\gamma) = \sum_{i=2}^{n} t_{x_i},
	\]
	and for vertex sets $A,B\subseteq \Z^2$, we define the first-passage time from $A$ to $B$ by
	\[
	T(A,B) = \inf \{ T(\gamma): \gamma \text{ is a path from a vertex in $A$ to a vertex in $B$}\}.
	\]
	For notational simplicity, for $x\in\Z^2$ and $B\subseteq \Z^2$, we will write $T(x, B)$ for $T(\{x\}, B)$. In first-passage percolation, one studies the asymptotic behavior of random variables such as $T(0,\partial B(n))$, where $B(n) = \{x\in\Z^2: \|x\|_\infty \leq n\}$ and $\partial B(n) = \{x\in\Z^2: \|x\|_\infty = n\}$.
	
	In this paper, we will study the critical case, namely $F(0) = p_c = 1/2$, where $p_c$ is the critical threshold for site percolation on $\T$. In this case, it is shown by Damron-Lam-Wang in \cite{critical} that under suitable moment assumptions on $t_x$, one has
	\begin{equation}
	\label{eq: asy}
	\E T(0,\partial B(n)) \asymp \sum_{k=2}^{\lfloor \log{n} \rfloor} F^{-1}(p_c+2^{-k}) \text{ and } \Var(T(0,\partial B(n))) \asymp \sum_{k=2}^{\lfloor \log{n} \rfloor} F^{-1}(p_c+2^{-k})^2,
	\end{equation}
	and further if $\Var(T(0,\partial B(n)))\to\infty$ as $n\to \infty$, then one also has a Gaussian central limit theorem for the variables $(T(0,\partial B(n)))$\footnote{These results were proved for edge-FPP on the square lattice, but similar arguments give them for the current setting.}. Sharper asymptotics were proved by C.-L. Yao in \cite{YaoLLN, Yao} (with further development in \cite{YaoBoundary,YaoAsymptotics}) in the special case where $t_x$ is Bernoulli (that is, $t_x = 0$ with probability $1/2$ and $t_x=1$ with probability $1/2$):
	\[
	\frac{T(0,\partial B(n))}{\log{n}} \to \frac{1}{2\sqrt{3}\pi} \text{ a.s., } \frac{\E T(0,\partial B(n))}{\log{n}} \to \frac{1}{2\sqrt{3}\pi}, \frac{\Var(T(0,\partial B(n)))}{\log{n}} \to \frac{2}{3\sqrt{3}\pi} - \frac{1}{2\pi^2}
	\]
	as $n\to\infty$. By analogy with the non-critical case of FPP, we will refer to the limit on the left (of $T(0,\partial B(n))/\log n$) as the {\it time constant} for the model. In \cite[Remark~1.3]{YaoLLN}, Yao asks whether one can extend these limit theorems to general distributions.
	
	The behaviors in \eqref{eq: asy} suggest that the limits in the limit theorems, if existent, should also depend on the behavior of $F^{-1}$ near $p_c$, or equivalently the behavior of $F$ near $0$. We will show existence of these limits and, from their explicit forms, it is manifest that this is indeed the case. Let 
	\[
	I = \inf\{x>0: F(x) > p_c\}
	\]
	be the infimum of the support of the law of $t_x$ excluding $0$.
	\begin{thm}
		\label{thm:lln}
		On the triangular lattice $\T$, we have a law of large numbers:
		\begin{equation}
		\label{eq: lln}
		\lim_{n\to\infty} \frac{T(0,\partial B(n))}{\log{n}} = \frac{I}{2\sqrt{3}\pi}\quad\text{almost surely}.
		\end{equation}
		Furthermore, if $\E \min\{t_1,\ldots,t_6\}^2<\infty$, where $t_1,\ldots, t_6$ are i.i.d. copies of $t_v$, then
		\begin{equation}
		\label{eq: var}
		\lim_{n\to\infty} \frac{\Var(T(0,\partial B(n)))}{\log{n}} = I^2\left(\frac{2}{3\sqrt{3}\pi} - \frac{1}{2\pi^2}\right).
		\end{equation}
	\end{thm}

	\begin{rem}
	\begin{enumerate}
		\item One can also show that if $\E \min\{t_1,\ldots,t_6\}<\infty$, then
		\begin{equation}
		\label{eq: mean}
		\lim_{n\to\infty}\frac{\E T(0,\partial B(n))}{\log{n}} = \frac{I}{2\sqrt{3}\pi}.
		\end{equation}
		This can be proved by using a similar, but simpler, method than that for \eqref{eq: var}. We omit the details here.
		\item  Our proof method extends to a large class of two-dimensional lattices (including the square lattice). It gives a weaker result, since the limits \eqref{eq: lln}, \eqref{eq: var} and \eqref{eq: mean} are only known to hold for the Bernoulli distribution on the triangular lattice (due to the use of CLE$_6$ in their proofs). However, if any of these are shown to exist on other lattices (with possibly different limits), then our method shows that they also hold for general distributions (under suitable moment assumptions). In implementing this, one would need to replace the exact values of arm exponents used in Lemma~\ref{lem: G_k_lemma} by inequalities for these exponents on general lattices given in \cite{KZexponents}.
		\item In the standard case of FPP, where $F(0)<p_c$, there is no similar universality of the time constant. Indeed, using \cite[Theorem~(2.13)]{BK}, one can construct two bounded distributions for $t_x$ such that they have the same infimum, but the limits $\lim_{n \to \infty} \frac{T(0,\partial B(n))}{n}$ for the different distributions are positive and distinct real numbers. 
		\item The moment conditions in Theorem~\ref{thm:lln} and equation \eqref{eq: mean} are optimal in the following senses. By a variant of \cite[Lemma~3.1]{coxdurrett}, one has for any $q > 0$, $\mathbf{E}T(0,\partial B(n))^q < \infty$ if and only if $\mathbf{E}\min\{t_1, \dots, t_6\}^q<\infty$. Therefore if the above moment conditions fail, then either the mean or the variance of $T(0,\partial B(n))$ will be infinite. Under a slightly stronger moment condition on $t_v$, one can prove point-to-point analogues of the statements of Theorem~\ref{thm:lln} (replacing $T(0,\partial B(n))$ with $T(0,x)$ and $\log n$ by $\log \|x\|$). See a similar modification in \cite{Yao}.
	\end{enumerate}
	\end{rem}

\paragraph{Question:}
	According to \eqref{eq: asy}, there exist distributions such that $\E T(0,\partial B(n)) = o(\log{n})$ and $\Var(T(0,\partial B(n))) = o(\log{n})$, but both quantities diverge to infinity as $n\to\infty$. In this case, does
	\[
	\lim_{n\to\infty} \frac{T(0,\partial B(n))}{\sum_{k=2}^{\lfloor \log n \rfloor} F^{-1}(p_c+2^{-k})} \text{ exist?}
	\]

\bigskip
\bigskip
In this paper, the symbol $C_i$ (where $i\in\N$) denotes a (possibly) large constant, and the symbol $c_i$ ($i\geq 4$) denotes a (possibly) small constant. $c_1, c_2, c_3$ are reserved for Definition~\ref{def: low_weight} and the definition of good circuit. The symbol $\|\cdot\|$ will refer to the Euclidean norm.

	\subsection{Sketch of proofs}	
	
	\subsubsection{Sketch of \eqref{eq: lln}}


	We begin by coupling together our vertex weights with Bernoulli weights: we define the Bernoulli weights as $t_v^\textnormal{B} = I \cdot \mathbf{1}_{\{t_v > 0\}}$. Because $t_v^\textnormal{B} \leq t_v$, one has
	\[
	\frac{I}{2\sqrt{3}\pi} = \lim_{n \to \infty} \frac{T^\textnormal{B}(0,\partial B(n))}{\log n} \leq \liminf_{n \to \infty} \frac{T(0,\partial B(n))}{\log n} \text{ almost surely}.
	\]
	(Here, $T^\textnormal{B}$ is the passage time using the Bernoulli weights.) 
	
	To show the other inequality, it will suffice to prove that
	\[
	T(0,\partial B(n)) - T^\textnormal{B}(0,\partial B(n)) = o(\log n) \text{ almost surely}.
	\]
	The idea for this proof is to use that $T^\textnormal{B}$ is equal to the maximal number of disjoint closed (that is, with weight $>0$) circuits separating 0 from $\partial B(n)$. To construct such circuits, we note in Lemma~\ref{lem: outermost} that results of \cite{KSZ} allow us to find an infinite sequence of disjoint closed circuits surrounding the origin which are successively ``outermost.'' In particular, it is not possible to find a closed circuit lying entirely in the region strictly between two adjacent circuits in the sequence. Because of this extremal property, one can construct an infinite geodesic $\Pi$ (a path all of whose finite segments are geodesics) for the weights $(t_v^\textnormal{B})$ by starting at 0, following any open (that is, with weight $=0$) path from 0 to the first circuit, using a vertex from this circuit, following any open path to the next circuit, and so on. One can show that if $\Pi_n$ is the portion of $\Pi$ until its first intersection with $\partial B(n)$, then
	\[
	T^\textnormal{B}(\Pi_n) - T^\textnormal{B}(0, \partial B(n)) = o(\log n) \text{ almost surely}.
	\]
	The goal then is to show that
	\begin{equation}\label{eq: to_show_sketch_1}
	T(\Pi_n) - T^\textnormal{B}(\Pi_n) = o(\log n) \text{ almost surely}
	\end{equation}
	for a particular choice of $\Pi$.
	
	To choose $\Pi$, we use an adaptation of the ``good circuit'' construction from \cite{KSZ}. Given an $\epsilon>0$, we consider a vertex $v$ to be of ``low weight'' if $t_v \leq I+\epsilon$. (Low-weight vertices are defined more precisely in Definition~\ref{def: low_weight}.) In Section~\ref{sec: construction}, we show that with high probability, any open path starting at one circuit in the above sequence and ending at the next --- so long as these circuits have sufficient distance from each other --- can be modified to retain the same initial point, but to end adjacent to a vertex in the second circuit which has low weight. (See Figure~\ref{fig: goodcircuits}.) This idea underlies the main construction (Lemma~\ref{lem: induction}) in which we build an infinite path $\gamma$ starting at 0 which passes through each circuit exactly once and contains only finitely many vertices which are not of low weight. Because $\gamma$ is an infinite geodesic for $T^\textnormal{B}$, we take $\Pi = \gamma$, so that $\Pi_n$ is an initial segment $\gamma_n$ of $\gamma$. Then
	\begin{align*}
	T(\gamma_n) - T^\textnormal{B}(\gamma_n) = \sum_{v : t_v > 0, v \in \gamma_n} (t_v - I) &\leq C+ \sum_{v : t_v > 0, v \in \gamma_n} \epsilon \\
	&= C+(\epsilon/I)T^\textnormal{B}(\gamma_n) \leq C\epsilon \log n.
	\end{align*}
	This is true for any $\epsilon$, so this shows \eqref{eq: to_show_sketch_1} and completes the sketch.

	\begin{figure}[h]
		\centering
		\includegraphics[width=.6\textwidth]{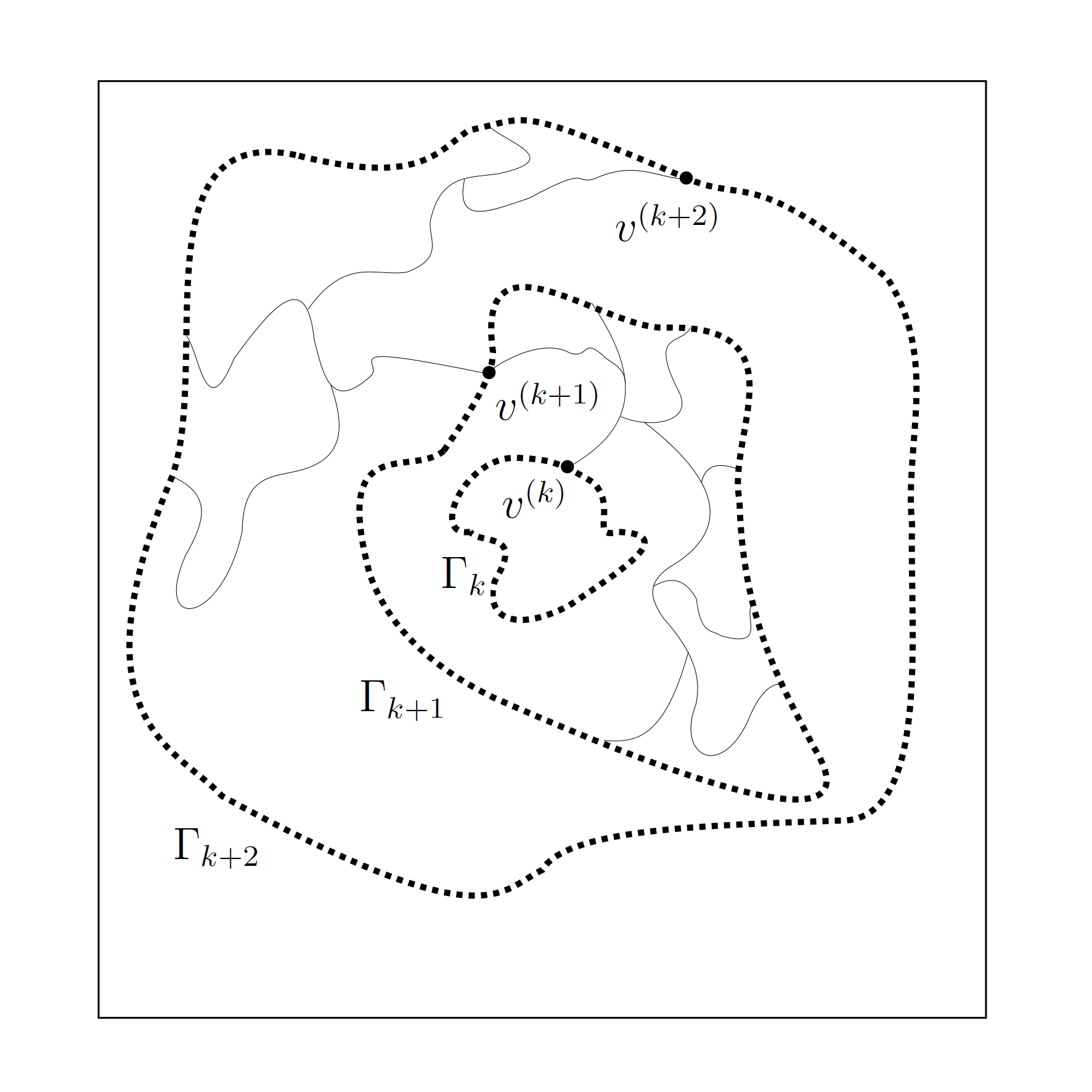}
		\caption{Illustration of the modification of open paths from Lemma~\ref{lem: induction}. The circuits $\Gamma_{k+i}$, $i=0, 1, 2$ are consecutive outermost circuits from the construction of Lemma~\ref{lem: outermost}. The light gray paths connecting vertices $v^{(i)}$ on the circuits are open. So long as the circuits have sufficient distance from each other, one can find many modified open paths which begin at the same point and end at the next circuit. With high probability, at least one such path will end at a low-weight vertex.}
		\label{fig: goodcircuits}
	\end{figure}

	\subsubsection{Sketch of \eqref{eq: var}}
	
	To show universality of the variance, we represent the passage time as a sum of martingale differences $\Delta_k$, so that
	\[
	\Var(T(0,\mathcal{O}_n)) = \sum_{k=0}^{n} \E \Delta_k^2,
	\]
	where $\mathcal{O}_n$ is the innermost open circuit in any annulus of the form $B(2^{m+1}) \setminus B(2^m)$ for $m \geq n$. Here, $\Delta_k$ is the martingale difference of $T(0, \mathcal{O}_n)$ using the filtration generated by weights on and inside $\mathcal{O}_k$. Using a representation of $\Delta_k$ from \cite{KestenZhang}, we split the difference
	\[
		\Var(T(0,\mathcal{O}_n)) - \Var(T^\textnormal{B}(0, \mathcal{O}_n)) = \sum_{k=0}^n \E(\Delta_k^2 - (\Delta_k^\textnormal{B})^2)
	\]
	into a sum of three terms, each of which is bounded similarly. The term we focus on can be written (see Lemma~\ref{lem: y_tilde}, where the term we are discussing is called $Y$) as a difference of passage times between two open circuits:
	\[
	T(\mathcal{O}_k,\bar{\mathcal{O}}_k) - T^\textnormal{B}(\mathcal{O}_k,\bar{\mathcal{O}}_k).
	\]
	(In Lemma~\ref{lem: y_tilde}, this difference is called $\tilde Y$.) Here $\bar{\mathcal{O}}_k$ is the next circuit of the form $\mathcal{O}_m$ which is not equal to $\mathcal{O}_k$. Therefore the proof reduces to showing
	\begin{equation}\label{eq: to_show_sketch_2}
	\sum_{k=0}^n \E(T(\mathcal{O}_k,\bar{\mathcal{O}}_k) - T^\textnormal{B}(\mathcal{O}_k,\bar{\mathcal{O}}_k))^2 = o(n).
	\end{equation}
	Once this is done, then the proof is completed by bounding the difference between point-to-box passage times of the form $T(0,\partial B(2^n))$ and point-to-circuit passage times of the form $T(0,\mathcal{O}_n)$. Although such bounds have been derived in previous works under stronger moment assumptions, the situation here is more delicate. We give this argument after \eqref{eq: var_diff}.
	
	To bound the terms of \eqref{eq: to_show_sketch_2}, we use the construction of low-weight paths from Section~\ref{sec: construction}. To use this method, we need to show in Lemma~\ref{lem: G_k_lemma} that with high enough probability, the circuit $\mathcal{O}_k$ is sufficiently far away from closed circuits $\mathcal{C}_j$ from the sequence in Lemma~\ref{lem: outermost}. When this occurs, one can, as in the proof of universality of the time constant, bound the difference of passage times in \eqref{eq: to_show_sketch_2} using paths connecting these circuits which pass through some number of closed circuits using only low-weight vertices. Partitioning the expectation according to this ``sufficiently far'' event, one has a bound for the summands of \eqref{eq: to_show_sketch_2} of the form
	\[
	o(1) + \epsilon \E (\text{maximal }\# \text{ of closed circuits between } \mathcal{O}_k \text{ and } \bar{\mathcal{O}}_k)^2.
	\]
	Here, the $o(1)$ term corresponds to the bound $C/\sqrt{k}$ in \eqref{eq: first_term}, and represents the expectation on the event that $\mathcal{O}_k$ is close to a closed circuit $\mathcal{C}_j$. The second term appears in \eqref{eq: come_back_here!} (where $\epsilon$ is written as a $k$-dependent term $a_k^2$). After showing that this expectation is bounded by a constant, we obtain the bound $o(1)+C \epsilon$ for summands in \eqref{eq: to_show_sketch_2}, and this completes the sketch.


	\section{Preliminaries}
	For a circuit $\Gamma$, we define $\mathring{\Gamma}$ to be the interior of $\Gamma$, namely the bounded connected component of $\Z^2 \setminus \Gamma$ seen as a subgraph of $\T$. 
	We say that a vertex $x$ is open if $\omega_x \leq 1/2$ and closed otherwise. A path (or circuit) is open if all its vertices are open; it is closed if all its vertices are closed.
	
		We will use several (modified) lemmas from \cite{KSZ}. The first provides an infinite sequence of ``outermost'' closed circuits surrounding 0.

	\begin{lem}
		\label{lem: outermost}
		Almost surely, there exists a sequence of random disjoint circuits $(\mathcal{C}_k)_{k\geq 1}$ with $0 \in \mathring{\mathcal{C}_k}$, so that each of these circuits is closed, $\mathcal{C}_k \subseteq \mathring{\mathcal{C}}_{k+1}$, and $\mathcal{C}_k$ is the  outermost circuit in $\mathring{\mathcal{C}}_{k+1}$ which is entirely closed. (Also there is no closed circuit surrounding $0$ in $\mathring{\mathcal{C}}_1$.) Moreover, there exist constants $C_1, C_2, c_4>0$ such that almost surely, $\diam(\mathcal{C}_{k+1}) \leq k^{C_1}\diam(\mathcal{C}_k)$ for all large $k$, and
		\begin{equation}\label{eq: outermost_inequality}
		c_4\leq \liminf_{k\to\infty} \frac{\log(\diam(\mathcal{C}_k))}{k} \leq \limsup_{k\to\infty} \frac{\log(\diam(\mathcal{C}_k))}{k} \leq C_2.
		\end{equation}
	\end{lem}
	\begin{proof}
	This statement is the same as that of \cite[Lemma~1]{KSZ}, except that in that paper, the circuits may be open or closed, and there is no mention of there being no circuit surrounding 0 in $\mathring{\mathcal{C}}_1$. The proof, however, is the same.
	\end{proof}
	
	The next lemma controls the number of circuits from the above sequence which intersect fixed boxes. It is a combination of \cite[Lemma~6]{KSZ} and the inequality in its proof (second paragraph in \cite[p.~23]{KSZ}).
	
	\begin{lem}
		\label{lem:squares}
		For $c\in (0,1)$, $j\geq 1$ and $r,s\in \Z$, define
		\[
		\tau(r, s) = \tau(r, s; c, j) = [r2^{c(j+1)}, (r+3)2^{c(j+1)}] \times [s2^{c(j+1)}, (s+3)2^{c(j+1)}]
		\]
		and $N(j,c)$ to be the number of squares $\tau(r,s)$ which intersect $B(2^{j+1})$ and intersect two successive circuits $\mathcal{C}_k$ and $\mathcal{C}_{k+1}$ with $\diam(\mathcal{C}_k) \geq 2^j$. Then for fixed $c\in (0,1)$, there exists $C_3>0$ such that
		\begin{equation}
		\label{eq: Njc_prob}
		\prob(N(j,c) > j^2 ) \leq \frac{C_3}{j^2}.
		\end{equation}
		In particular, for fixed $c\in (0,1)$, almost surely, 
		\begin{equation}
		\label{eq:Njc}
		N(j,c) \leq j^2 \quad \text{for all large $j$}.
		\end{equation}
		Moreover, if $N^{(3)}(j,c)$ is the number of squares $\tau(r, s)$ which intersect $B(2^{j+1})$ and intersect three successive circuits $\mathcal{C}_k$, $\mathcal{C}_{k+1}$, $\mathcal{C}_{k+2}$ with $\diam(\mathcal{C}_k) \geq 2^{j-(\log{j})^2}$, then there exists $C_4>0$ such that
		\begin{equation}
		\label{eq: N3_prob}
		\prob(N^{(3)}(j,c) > 0) \leq \frac{C_4}{j^2}.
		\end{equation}
		In particular, almost surely,
		\begin{equation}
		\label{eq:N3}
		N^{(3)}(j,c) = 0 \quad \text{for all large $j$}.
		\end{equation}
	\end{lem}

	The following deterministic lemma is a simplified version of  \cite[Lemma~7]{KSZ}. For any sets of vertices $S,S'$, we write $d(S,S')$ for $\min\{\|x-y\| : x \in S, y \in S'\}$.
	\begin{lem}
		\label{lem:separate}
		Let $\mathcal{C}_{k+1}$ and $\mathcal{C}_{k+2}$ be two successive circuits from the sequence of Lemma~\ref{lem: outermost}. Assume that $2^j \leq \diam(\mathcal{C}_{k+1}) < 2^{j+1}$. Let $v_1^{(k+1)}, v_2^{(k+1)}, \ldots, v_M^{(k+1)}$ be $M$ arbitrary vertices of $\mathcal{C}_{k+1}$ for which
		\begin{equation}
		\|v_p^{(k+1)} - v_q^{(k+1)}\| \geq 8\cdot 2^{cj},  \text{ for all } p,q \text{ with } p\neq q.
		\end{equation}
		If $N(j, c)\leq j^2$, then at least $M-j^2$ of the vertices $v_m^{(k+1)}$ satisfy $d(v_m^{(k+1)}, \mathcal{C}_{k+2}) > 2^{cj}$.
	\end{lem}

	We will also define ``good circuits,'' which will allow us in Section~\ref{sec: construction} to construct an infinite path starting at zero whose intersection with those circuits have low weights. 
	Let $\Gamma$ be a circuit surrounding the origin with $2^j \leq \diam(\Gamma) < 2^{j+1}$. For constants $c_1, c_3\in (0,1)$, we consider open connected sets $\mathcal{D}$ with some or all of the following properties:
	\begin{equation}
	\label{eq:4.1}
	\textnormal{$\mathcal{D} \subseteq \mathring{\Gamma}$ and $\mathcal{D}$ contains exactly one vertex adjacent to $\Gamma$;}
	\end{equation}
	\begin{equation}
	\label{eq:4.2}
	\diam(\mathcal{D}) \geq (\diam(\Gamma))^{c_1};
	\end{equation}
	\begin{align}
	\label{eq:4.3}
	&\textnormal{the open cluster of $\mathcal{D}$ in $\mathring{\Gamma}$ contains at least $(\diam(\Gamma))^{c_2}$ vertices} \nonumber\\
	&\textnormal{$w_m$ which are adjacent to $\Gamma$ and satisfy $\|w_p-w_q\| \geq (\diam(\Gamma))^{c_1c_3}$} \nonumber\\
	&\textnormal{for $p\neq q$; moreover, there exists a vertex $z \in \mathcal{D}$ and for each of} \nonumber\\
	&\textnormal{the $w_m$ an open path from $z$ to $w_m$ such that only its endpoint} \nonumber\\
	&\textnormal{$w_m$ is adjacent to $\Gamma$.}
	\end{align}
	Here, the open cluster of $\mathcal{D}$ in $\mathring{\Gamma}$ is the largest open connected set in $\mathring{\Gamma}$ that contains $\mathcal{D}$. 
	We say a closed circuit $\Gamma$ is $(c_1, c_2, c_3)$-good if any open connected set $\mathcal{D}$ with \eqref{eq:4.1} and \eqref{eq:4.2} also satisfies \eqref{eq:4.3}. Note that if $\Gamma$ is $(c_1,c_2,c_3)$-good, then for any $\hat c_2 \leq c_2$, it is also $(c_1,\hat c_2,c_3)$-good.
	\begin{lem}
	\label{lem:good}
	For any $c_1, c_3\in (0,1)$ with $c_1$ sufficiently close to $1$, we can choose $c_2\in (0,1)$, depending only on $c_1$ and $c_3$, such that
	\begin{equation}
	\label{eq: good_circuits}
	\prob(\exists \textnormal{ a circuit $\mathcal{C}_k$ with $2^j \leq\diam(\mathcal{C}_k) < 2^{j+1}$ which is not $(c_1, c_2, c_3)$-good}) \leq C_5 e^{-c_5j},
	\end{equation}
	where $C_5, c_5$ depend only on $c_1, c_3$. Moreover, given such $c_i$'s, we can choose $c_1', c_2', c_3' \in (0,1)$ with $c_1'c_3'> c_1$ and $c_2'=c_2$ such that \eqref{eq: good_circuits} holds with $c_i$ replaced by $c_i'$, $i=1,2,3$.
	\end{lem}
	\begin{proof}
		In \cite{KSZ}, the authors define a circuit $\Gamma$ to be good if any open connected set $\mathcal{D}$ with \eqref{eq:4.1} and \eqref{eq:4.2} also satisfies the following (instead of \eqref{eq:4.3}):
		\begin{itemize}
			\item[] the open cluster of $\mathcal{D}$ in $\mathring{\Gamma}$ contains at least $(\diam(\Gamma))^{c_2}$ self-avoiding paths $\theta_m$ which are adjacent to $\Gamma$, have length $\geq c_6\log\log(\diam(\Gamma))$ (where $c_6>0$ is a constant) and satisfy $d(\theta_p, \theta_q) \geq (\diam(\Gamma))^{c_1c_3}$ for $p\neq q$; moreover, there exists a vertex $z\in\mathcal{D}$ and for each of the $\theta_m$ an open path from $z$ to $\theta_m$ such that only its endpoint on $\theta_m$ is adjacent to $\Gamma$.
		\end{itemize}
		Since the above requires more than \eqref{eq:4.3}, by \cite[Proposition~1]{KSZ}, we obtain \eqref{eq: good_circuits}.
	\end{proof}

	\section{Construction of a low-weight path}\label{sec: construction}
	Let $(\mathcal{C}_i)$ be fixed as $(\Gamma_i)$ for some sequence of circuits with $\Gamma_i \subseteq \mathring{\Gamma}_{i+1}$. We would like to construct a self-avoiding path from $0$ such that (except for finitely many vertices) it contains only open vertices or some $v^{(k)}\in \Gamma_k$ with $v^{(k)}$ being of low weight. For the definition of low weight, let $c_2\in (0,1)$ (which will be taken to be the same as the $c_2$ in \eqref{eq:4.3}) and recall that $I = \inf\{x>0: F(x) > p_c\}$.
	\begin{df}\label{def: low_weight}
		Let $\Gamma$ be a circuit such that $0 \in \mathring{\Gamma}$ with $2^j\leq \diam(\Gamma) <2^{j+1}$.
		\begin{enumerate}
			\item If $\prob(t_v = I) > 0$ and $I>0$, we say that a vertex $v\in \Gamma$ is of ($j$-)low-weight if $t_v = I$ and we define $a_j=0$ for all $j$.
			\item If $\prob(t_v = I) = 0$ or $I=0$, we fix any non-increasing sequence $(a_j)$ such that $\prob(I < t_v \leq I+a_j) \geq 2^{-c_2j/2 - 1}$ and $a_j\to 0$ as $j\to\infty$, and we say that a vertex $v\in \Gamma$ is of ($j$-)low-weight if $I < t_v \leq I+a_j$.
		\end{enumerate}
	\end{df}

	For the following statement, given a configuration of open/closed vertices, let $\bar{\prob}$ be the (regular) conditional distribution of the variables $(\omega_x)$ given this configuration. 
	\begin{lem}
		\label{lem: induction}
		Choose $c_1,c_3$ and $c_1',c_3'$ in $(0,1)$ with corresponding $c_2=c_2'$ as dictated by Lemma~\ref{lem:good}, and both $c \in (c_1,c_1'c_3')$ and $\hat c \in (0,c_1c_3)$. Also fix $C_6, C_7, c_7>0$. There exists $c_{8}>0$ such that the following holds for all sufficiently large $k$. For a given configuration $\eta$ of open/closed vertices, suppose $2^j \leq \diam(\Gamma_{k+1}) < 2^{j+1}$ and $v^{(k)} \in \Gamma_k$ satisfies
  \begin{equation}\label{eq: initial_assumption}
  d(v^{(k)}, \Gamma_{k+1}) \geq 2(\diam(\Gamma_{k+1}))^{c_1'}.
  \end{equation}
Assume that
  \begin{itemize}
  \item $N(i,c) \leq i^2$, $N(i,\hat c) \leq i^2$, and $N^{(3)}(i,c_1') = 0$ for all $i \geq j$,
  \item one has
  \begin{equation}\label{eq:circuit_size}
  \diam(\Gamma_{i+2}) \leq (i+1)^{C_6}\diam(\Gamma_{i+1}) \text{ for all }i \geq k,
  \end{equation}
  \item one has
  \begin{equation}\label{eq: growth_bound}
  c_7i \leq \log(\diam(\Gamma_{i+2})) \leq C_7i \text{ for all }i \geq k,
  \end{equation}
  \item $\Gamma_{i+1}$ and $\Gamma_{i+2}$ are $(c_1,c_2,c_3)$- and $(c_1',c_2',c_3')$-good for all $i \geq k$.
  \end{itemize}
  With $\bar{\prob}$-probability at least $1-e^{-c_8j}$, (conditioned on $\eta$) we can find sequences $(v^{(i)})_{i \geq k+1}$ and $(\mathcal{D}^{(i)})_{i \geq k}$ such that for all $i$,
  \begin{enumerate}
  \item for $i \geq k$, $v^{(i)} \in \Gamma_i$ and $\mathcal{D}^{(i)}$ is an open path connecting a neighbor of $v^{(i)}$ with a neighbor of $v^{(i+1)}$;
  \item for $i \geq k+1$, $v^{(i)}$ is of low weight; and
  \item for $i \geq k$, $\mathcal{D}^{(i)}$ contains only one vertex adjacent to $\Gamma_{i+1}$.
  \end{enumerate}
	\end{lem}
	
	\begin{proof}
	The proof is similar to the construction of double paths in \cite[Section~5]{KSZ}. We will first construct $v^{(k+1)}$ and $\mathcal{D}^{(k)}$. Because $v^{(k)} \in \Gamma_k$, there exists an open path $\hat{\mathcal{D}}^{(k)}$ from a neighbor of $v^{(k)}$ to a neighbor of $\Gamma_{k+1}$ that only contains one vertex adjacent to $\Gamma_{k+1}$. Due to \eqref{eq: initial_assumption}, we have
	\[
	\diam(\hat{\mathcal{D}}^{(k)}) \geq 2(\diam(\Gamma_{k+1}))^{c_1'} - 4 \geq (\diam(\Gamma_{k+1}))^{c_1'}.
	\]
Since $\Gamma_{k+1}$ is $(c_1', c_2', c_3')$-good, there are at least $(\diam(\Gamma_{k+1}))^{c_2'} \geq 2^{c_2'j}$ many vertices $w_m$ in the open cluster of $\hat{\mathcal{D}}^{(k)}$ which are adjacent to $\Gamma_{k+1}$ and satisfy $\|w_p-w_q\| \geq (\diam(\Gamma_{k+1}))^{c_1'c_3'}$ if $p\neq q$. This implies (for $k$ large) $\|w_p - w_q\| \geq 16(\diam(\Gamma_{k+1}))^c$ if $p\neq q$. So if we choose $(v_m^{(k+1)})_m$ as vertices in $\Gamma_{k+1}$ adjacent to the $w_m$'s (in some deterministic and $\eta$-measurable way), then if $p \neq q$,
	\[
	\|v_p^{(k+1)}-v_q^{(k+1)}\| \geq 8(\text{diam}(\Gamma_{k+1}))^c \geq 8 \cdot 2^{cj}.
	\]
	Since $N(c,j) \leq j^2$, by Lemma~\ref{lem:separate}, at least $2^{c_2'j}-j^2$ of the $v_m^{(k+1)}$'s satisfy $d(v_m^{(k+1)},\Gamma_{k+2}) > 2^{cj}$. We claim that with conditional probability at least $1-e^{-c_{9}j}$, more than $j^2$ of the $v_m^{(k+1)}$'s have low weight. To see why, we use the Chernoff bound: letting $X_1, \dots, X_r$ be i.i.d. Bernoulli random variables with parameter $2^{-c_2'j/2}$, and $r=2^{c_2'j}$, then
	\begin{align*}
	\bar{\prob}(\text{at most }j^2 \text{ of the } t_{v_m^{(k+1)}}\text{'s have low weight}) &\leq \prob(X_1 + \dots + X_r \leq j^2) \\
	&= \prob(\exp(-(X_1 + \dots + X_r)) \geq e^{-j^2}) \\
	&\leq e^{j^2} \left[ \mathbf{E} \exp(-X_1)\right]^r \\
	&=e^{j^2} \left[ (1-2^{-c_2'j/2}) + e^{-1}2^{-c_2'j/2}\right]^{2^{c_2'j}} \\
	&\leq e^{-c_{10} 2^{c_{11}j}}.
	\end{align*}
	(Here, we have assumed that item 2 of Definition~\ref{def: low_weight} holds; otherwise, the proof is even easier.) This shows the claim. When it holds (that is, more than $j^2$ of the vertices have low weight), we say that ``the first stage is successful.'' Note that given $\eta$, the outcome of the first stage depends only on the weights for vertices on $\Gamma_{k+1}$.
	
	Assuming that the first stage is successful, we can choose $v^{(k+1)} \in \{v_m^{(k+1)}\}$ such that both conditions hold:
\[
	d(v^{(k+1)}, \Gamma_{k+2}) > 2^{cj} \text{ and } v^{(k+1)} \text{ is of low-weight}.
\]
	This implies by \eqref{eq:circuit_size} that for large $k$,
	\begin{equation}\label{eq: boston_lasagna}
	d(v^{(k+1)}, \Gamma_{k+2}) \geq 2(\diam(\Gamma_{k+2}))^{c_1}.
	\end{equation}
	Last, we define $\mathcal{D}^{(k)}$ by modifying $\hat{\mathcal{D}}^{(k)}$ so that it begins at a neighbor of $v^{(k)}$ and ends at a neighbor of $v^{(k+1)}$ (with only one vertex adjacent to $\Gamma_{k+1}$).
	
	Now we construct the further paths $\mathcal{D}^{(k+1)}, \mathcal{D}^{(k+2)}, \dots$ and $v^{(k+2)},v^{(k+3)}, \dots$. Inequality \eqref{eq: boston_lasagna} means we can find an open path $\hat{\mathcal{D}}^{(k+1)}$ connecting a neighbor of $v^{(k+1)}$ with a neighbor of $\Gamma_{k+2}$, with only one vertex adjacent to $\Gamma_{k+2}$, such that
	\[
	\diam(\hat{\mathcal{D}}^{(k+1)}) \geq 2(\diam(\Gamma_{k+2}))^{c_1}-4 \geq (\diam(\Gamma_{k+2}))^{c_1}.
	\]
	Therefore we can repeat the argument leading to \eqref{eq: boston_lasagna} with $\hat{\mathcal{D}}^{(k+1)}$ in place of $\mathcal{D}^{(k)}$, using now that $\Gamma_{k+2}$ is $(c_1,c_2,c_3)$-good (and putting $\hat c$ in place of $c$), to construct an open path $\mathcal{D}^{(k+1)}$ from a neighbor of $v^{(k+1)}$ to a neighbor of some $v^{(k+2)} \in \Gamma_{k+2}$ of low weight that has only one vertex adjacent to $\Gamma_{k+2}$. Again, we will only be able to do this with conditional probability $\geq 1-\exp\left(-c_{9}\lfloor\log_2 \diam(\Gamma_{k+2})\rfloor\right)$, given both $\eta$ and the outcome of the first stage. (Note that conditioning on the outcome of the first stage only gives information about the weights on $\Gamma_{k+1}$.) If we are able to find such $\hat{\mathcal{D}}^{(k+1)}$ and $v^{(k+2)}$, then we say that ``the stage 2a is successful.''
	
	Unfortunately the argument above only gives $d(v^{(k+2)},\Gamma_{k+3}) \geq 2(\diam(\Gamma_{k+3}))^\alpha$ for some $\alpha<c_1$ and this is not enough to iterate the argument (the estimate will continue to deteriorate at each further iteration). We now claim that we can choose $\mathcal{D}^{(k+1)}$ and $v^{(k+2)}$ such that
	\begin{equation}\label{eq: inductive_step}
	d(v^{(k+2)},\Gamma_{k+3}) \geq 2(\diam(\Gamma_{k+3}))^{c_1}.
	\end{equation}
	
	To show \eqref{eq: inductive_step}, we argue as follows. If it so happens that $\diam(\mathcal{D}^{(k+1)}) \geq (\diam(\Gamma_{k+2}))^{c_1'}$, then we repeat the argument leading to \eqref{eq: boston_lasagna} with $\mathcal{D}^{(k+1)}$ in place of $\mathcal{D}^{(k)}$ (and the same value of $c$) to produce yet another open path $\widetilde{\mathcal{D}}^{(k+1)}$ connecting a neighbor of $v^{(k+1)}$ with a neighbor of some $v^{(k+2)} \in \Gamma_{k+2}$ of low weight, with only one vertex adjacent to $\Gamma_{k+2}$, but this time we will have the estimate \eqref{eq: inductive_step} using $\widetilde{\mathcal{D}}^{(k+1)}$ in place of $\mathcal{D}^{(k+1)}$. The conditional probability that we can find such $\widetilde{\mathcal{D}}^{(k+1)}$ and $v^{(k+2)}$ is again at least $1-\exp\left(-c_{9}\lfloor \log_2 \diam(\Gamma_{k+2})\rfloor\right)$. (If we can find such a path and vertex, we say that ``the stage 2b is successful.'') Otherwise, we must have
	\[
	\diam(\mathcal{D}^{(k+1)}) < (\diam(\Gamma_{k+2}))^{c_1'},
	\]
	and so
	\begin{equation}
	\label{eq:successive_v}
	\|v^{(k+1)} - v^{(k+2)}\| \leq \diam(\mathcal{D}^{(k+1)}) + 4 < 2(\diam(\Gamma_{k+2}))^{c_1'}.
	\end{equation}
	In this case, we always declare stage 2b to be successful. If \eqref{eq: inductive_step} fails, for large $k$, we use \eqref{eq:circuit_size} to see that
	\[
	d(v^{(k+2)},\Gamma_{k+3}) < 2(\diam(\Gamma_{k+2}))^{c_1'}.
	\] 
	This, together with \eqref{eq:successive_v}, implies that each of $\Gamma_{k+1}$, $\Gamma_{k+2}$, $\Gamma_{k+3}$ must have a point in $v^{(k+2)} + B(2(\diam(\Gamma_{k+2}))^{c_1'})$. Letting $i$ be such that $2^i \leq \diam(\Gamma_{k+2}) < 2^{i+1}$, there exist $r$, $s$ so that
	\[
	v^{(k+2)} \in [r2^{c_1'(i+1)}, (r+1)2^{c_1'(i+1)}] \times [s2^{c_1'(i+1)}, (s+1)2^{c_1'(i+1)}].
	\]
	Therefore, if $q$ is chosen such that $c_1'q \geq 4$,
	\begin{align*}
	&\tau(\lfloor 2^{-c_1'q} (r-2)\rfloor, \lfloor 2^{-c_1'q} (s-2)\rfloor; c_1', i+q) \\
	&\supseteq [r2^{c_1'(i+1)}, (r+1)2^{c_1'(i+1)}] \times [s2^{c_1'(i+1)}, (s+1)2^{c_1'(i+1)}] + B(2\cdot 2^{c_1'(i+1)})
	\end{align*}
	would intersect $B(2^{i+1}) \subseteq B(2^{i+q+1})$, as well as $\Gamma_{k+1}$, $\Gamma_{k+2}$, and $\Gamma_{k+3}$. Since $\diam(\Gamma_{k+3}) \geq \diam(\Gamma_{k+2}) \geq 2^i$ and by \eqref{eq:circuit_size} and \eqref{eq: growth_bound}, $\diam(\Gamma_{k+1}) \geq (k+1)^{-C_6} \diam(\Gamma_{k+2}) \geq 2^{i+q - (\log{(i+q)})^2}$ for large $i$, we would have $N^{(3)}(i+q, c_1')\neq 0$, but this is impossible by the hypothesis. Therefore \eqref{eq: inductive_step} holds.
	
	At this point, we have constructed $\mathcal{D}^{(k)}, \mathcal{D}^{(k+1)}, v^{(k+1)},$ and $v^{(k+2)}$. The conditional probability that stages 1, 2a, 2b are all successful is at least
	\begin{align*}
	&\bar{\prob}(\text{stage 2b is successful} \mid \text{stages 2a and 1 are successful}) \\
	\times~& \bar{\prob}(\text{stage 2a is successful} \mid \text{stage 1 is successful}) \\
	\times~& \bar{\prob}(\text{stage 1 is successful}) \\
	\geq~& \left(1 - \exp\left(-c_{9}\lfloor \log_2 \diam(\Gamma_{k+2})\rfloor\right)\right)^2 (1-e^{-c_{9}j}).
	\end{align*}
	
	Now that we have constructed $v^{(k+2)}$ and $\mathcal{D}^{(k+1)}$ such that \eqref{eq: inductive_step} holds, we can now repeat the argument we just gave that derived \eqref{eq: inductive_step} from \eqref{eq: boston_lasagna}, but with $v^{(k+2)},v^{(k+3)}$ in place of $v^{(k+1)},v^{(k+2)}$ and $\mathcal{D}^{(k+1)},\mathcal{D}^{(k+2)}$ in place of $\mathcal{D}^{(k)},\mathcal{D}^{(k+1)}$. From this, we reproduce the estimate \eqref{eq: inductive_step} with $v^{(k+3)},\Gamma_{k+4}$ in place of $v^{(k+2)},\Gamma_{k+3}$, so long as the corresponding steps 2a and 2b (which we will label 3a and 3b) are successful. The probability that these stages 3a and 3b are both successful, conditioned on the success of stages 1, 2a, and 2b, is at least
	\[
	\left( 1 - \exp\left(-c_{9}\lfloor \log_2  \diam(\Gamma_{k+3})\rfloor\right)\right)^2.
	\]
	Continuing in this way, we produce all paths $\mathcal{D}^{(i)}$ and vertices $v^{(i)}$ with conditional probability at least
	\begin{align*}
	&(1-e^{-c_9j}) \prod_{i=k}^\infty \left( 1- \exp\left( -c_{9}\lfloor \log_2 \diam(\Gamma_{i+2})\rfloor\right)\right)^2 \\
	\geq~&(1-e^{-c_{9}j}) \prod_{i=k}^\infty \left( 1- \exp\left( - c_{9}c_{7} i /2\right)\right)^2 \\
	\geq~&(1-e^{-c_{9}j}) (1-e^{-c_{12}k}) \\
	\geq~&1-e^{-c_{8}j}.
	\end{align*}
	Here, we have used \eqref{eq: growth_bound} to go from the first to second line, and then \eqref{eq: growth_bound} again, along with the inequality $\diam(\Gamma_{k+1}) \geq 2^j$, to go from the third to fourth line. This completes the proof.
	\end{proof}
	
%
%

	\begin{rem}\label{rem: construction}
	In the statement of Lemma~\ref{lem: induction}, the vertex $v^{(k)}$ is assumed to be on $\Gamma_k$ and to obey the distance bound \eqref{eq: initial_assumption}. It is straightforward to check that these conditions may be replaced by the following: there is an open path starting at a neighbor of $v^{(k)}$ of diameter $(\diam (\Gamma_{k+1}))^{c_1'}$ that contains only one vertex adjacent to $\Gamma_{k+1}$. (Here $v^{(k)}$ is not assumed to be on $\Gamma_k$.) In this case, the result holds with the same conditional probability bound: at least $1-e^{-c_{8}\lfloor \log_2 \diam(\Gamma_{k+1})\rfloor}$.
	\end{rem}

	\section{Universality of the time constant (asymptotic form)}
	In this section, we prove \eqref{eq: lln}. We define $(t_x^\textnormal{B})_{x \in \mathbb{Z}^2}$ to be a family of Bernoulli random variables with probability $1/2$ of being $0$ and probability $1/2$ of being $I$. We couple $t_x$ and $t_x^\textnormal{B}$ together using $\omega_x$: we can use $t_x^\textnormal{B} = I\cdot \mathbf{1}_{\{\omega_x > 1/2\}}$.  We also write $T^\textnormal{B}$ for the first-passage time using the weights $(t_x^\textnormal{B})$. By \cite[Proposition~3.6]{Yao},
	\begin{equation}\label{eq: YAO}
	\lim_{n\to\infty} \frac{T^\textnormal{B}(0,\partial B(n))}{\log{n}} = \frac{I}{2\sqrt{3}\pi}\quad\text{almost surely.}
	\end{equation}
	
	We now construct an infinite path using Lemma~\ref{lem: induction}, so choose $c_1,c_3$ and $c_1',c_3'$ with corresponding $c_2=c_2'$ as dictated by Lemma~\ref{lem:good}. Also fix $C_6,C_7, c_7>0$, $c \in (c_1,c_1'c_3')$, and $\hat c \in (0,c_1c_3)$. For $j,k \geq 1$, let $\Xi_{j,k}$ be the set of $\eta$ for which there exists $v^{(k)}$ such that the hypotheses of Lemma~\ref{lem: induction} hold for $\eta,j,k,v^{(k)}$. (In particular, $2^j \leq \diam(\mathcal{C}_{k+1}) < 2^{j+1}$.) By the lemma, the probability that there exists an infinite path $\gamma$ starting at some vertex $v^{(k)}$ of $\mathcal{C}_k$  (for $k$ fixed and large) that is open except for its intersection with each $\mathcal{C}_i$ ($i \geq k$), which consists of a low-weight vertex, is at least 
	\begin{align}
	\sum_{j=1}^\infty \prob(\Xi_{j,k}) (1-e^{-c_{8}j}) &\geq \sum_{j=c_4 k / 2}^\infty \prob(\Xi_{j,k})(1-e^{-c_{8}j}) \nonumber \\
	&\geq (1-e^{-c_{8}c_4k/2}) \sum_{j=c_4k/2}^\infty \prob(\Xi_{j,k}). \label{eq: chocolate_quest_bar}
	\end{align}
	 
	
	By Lemma~\ref{lem:squares}, almost surely, there exists a (random) integer $j$ such that 
	\[\max\{N(i,c),N(i,\hat c), N(i, c_1')\}\leq i^2
	 \text{ and } N^{(3)}(i,c_1') = 0 \text{ for all } i\geq j.\] Also, by a minor adaptation of \cite[Eq.~(5.25)]{KSZ} and the above Lemma~\ref{lem: outermost}, the probability that \eqref{eq:circuit_size} and \eqref{eq: growth_bound} hold goes to one as $k \to \infty$, so long as $C_6$ and $C_{7}$ are large enough and $c_{7}$ is small enough. Last, by Lemma~\ref{lem:good}, the probability tends to one as $k\to\infty$ that for all $i \geq k$, $\mathcal{C}_{i+1}$ and $\mathcal{C}_{i+2}$ are $(c_1,c_2,c_3)$- and $(c_1',c_2',c_3')$-good. As for condition \eqref{eq: initial_assumption}, if it fails for our $k$, then we must have $N(i,c_1') > i^2$ for some $i \geq (\log_2 k)-1$. This has probability tending to zero as $k\to\infty$. These facts, in conjunction with the fact that $(\Xi_{j,k})_j$ is a disjoint family for each $k$, shows that $\prob(\cup_j \Xi_{j,k}) \to 1$ as $k \to \infty$. Therefore, to show that \eqref{eq: chocolate_quest_bar} tends to 1 as $k\to\infty$, we must show that $\sum_{j=1}^{c_4k/2} \prob(\Xi_{j,k}) \to 0$ as $k\to\infty$. But by disjointedness again,
	 \[
	 \sum_{j=1}^{c_4k/2} \prob(\Xi_{j,k}) \leq \prob(\diam(\mathcal{C}_{k+1}) \leq 2^{c_4k/2 +1}) \to 0 \text{ as } k \to\infty
	 \]
	 by \eqref{eq: outermost_inequality}.
	We conclude that almost surely, there exists an infinite path $\gamma$ as described above, starting at a vertex $v^{(k)}$ of $\mathcal{C}_k$ for some (random) $k$.
	
	
	Fix a configuration $\omega$ for which $\gamma$ exists, and write $\gamma_n$ for the segment of $\gamma$ starting at $v^{(k)}$ and ending at the first intersection of $\gamma$ with $\partial B(n)$. ($\gamma_n$ is set to be empty if $v^{(k)} \notin B(n)$.) Then
	\begin{equation}\label{eq: hamburger_helper}
	\begin{split}
	T^\textnormal{B}(0,\partial B(n)) \leq T(0,\partial B(n)) &\leq T(0,v^{(k)}) + T(\gamma_n) \\
	&\leq T(0,v^{(k)}) + \sum_{\ell=1}^\infty (I + a(\ell))\mathbf{1}_{\{\mathcal{C}_\ell \cap B(n) \neq \emptyset\}},
	\end{split}
	\end{equation}
	where $a(\ell) = a_j$ if $2^j \leq \diam(\mathcal{C}_\ell) < 2^{j+1}$. Given $\epsilon>0$, since $a_j \leq \epsilon$ for $j$ large, the above is bounded by
	\[
	T(0,v^{(k)}) + \sum_{\ell=1}^{L} (I+a(\ell)) \mathbf{1}_{\{\mathcal{C}_\ell \cap B(n) \neq \emptyset\}} + (I+\epsilon) \#\{\ell : \mathcal{C}_\ell \cap B(n) \neq \emptyset\},
	\]
	where $L$ is some (random) finite number independent of $n$. We next use that $T^\textnormal{B}(0,\partial B(n))$ equals the maximal number of disjoint closed circuits surrounding 0 in $B(n)$ (see for instance \cite[Proposition~2.4]{Yao}) to bound this above by
	\begin{equation}\label{eq: almost_end!!}
	T(0,v^{(k)}) + \sum_{\ell=1}^L (I+a(\ell)) \mathbf{1}_{\{\mathcal{C}_\ell \cap B(n) \neq \emptyset\}} + (I+\epsilon)\left[ T^\textnormal{B}(0,\partial B(n)) + N_m\right],
	\end{equation}
	where $N_m$ is the maximal number of disjoint closed circuits surrounding 0 that intersect $B(2^{m+1}) \setminus B(2^{m})$, and $m = \lfloor \log_2 n\rfloor$. By the RSW theorem and the BK inequality (see \cite[Ch.~11]{grimmett}), one has
	\[
	\prob(N_m \geq K) \leq e^{-c_{13}K} \text{ for some }c_{13}>0 \text{ and all } K \geq 1,
	\]
	and so by the Borel-Cantelli lemma, $N_m \leq \log^2m$ for all large $m$, almost surely. Returning to \eqref{eq: almost_end!!}, for large $n$, we obtain the bound
	\[
	T(0,v^{(k)}) + \sum_{\ell=1}^L (I+a(\ell)) \mathbf{1}_{\{\mathcal{C}_\ell \cap B(n)\neq \emptyset\}} + (I+\epsilon)\left[ T^\textnormal{B}(0,\partial B(n)) + \log^2(\log_2 n)\right].
	\]
	Combining this with \eqref{eq: YAO} and \eqref{eq: hamburger_helper}, we obtain
	\[
	\frac{I}{2\sqrt{3}\pi} \leq \liminf_{n \to \infty} \frac{T(0,\partial B(n))}{\log n} \leq \limsup_{n \to \infty} \frac{T(0,\partial B(n))}{\log n} \leq \frac{I+\epsilon}{2\sqrt{3}\pi} \text{ almost surely}.
	\]
	As $\epsilon$ was arbitrary, this completes the proof.

	\section{Universality of limiting variance}
	In this section, we prove \eqref{eq: var}. We will use the martingale introduced in \cite{KestenZhang}. Define, for $k\geq 0$,
	\[
	A(k) = B(2^{k+1}) \setminus B(2^k),
	\]
	and
	\[
	m(k) = \inf\{n\geq k: A(n)\text{ contains an open circuit surrounding $0$}\}.
	\]
	We also define
	\[
	\mathcal{O}_k = \text{innermost open circuit surrounding $0$ in $A(m(k))$}
	\]
	for $k\geq 0$ and $\mathcal{O}_{-1} = \{0\}$. Finally, for $k\geq -1$, we define
	\[
	\mathcal{F}_k = \sigma\text{-field generated by $\left\{ \{\mathcal{O}_k = \Gamma\} \cap \{t_{x_1} \in A_1, \dots, t_{x_n} \in A_n\}\right\}_{\Gamma, x_i, A_i}$}
	\]
	for $\Gamma$ a circuit surrounding $0$ outside of $B(2^k)$, $x_i \in \overline{\Gamma}$, and $A_i \subseteq \mathbb{R}$ Borel. (This is the $\sigma$-field ``generated by the weights on and inside $\mathcal{O}_k$,'' and $\overline{\Gamma}$ refers to the union of $\Gamma$ and its interior $\mathring{\Gamma}$.)
	
	Since $\mathcal{O}_{k-1} \subseteq \overline{\mathcal{O}_k}$, we have $\mathcal{F}_{k-1}\subseteq \mathcal{F}_k$, and hence
	\begin{align*}
	T(0, \mathcal{O}_n) - \E T(0,\mathcal{O}_n) &= \sum_{k=0}^n \E [T(0,\mathcal{O}_n)|\mathcal{F}_k] - \E [T(0,\mathcal{O}_n)|\mathcal{F}_{k-1}]\\
	&=: \sum_{k=0}^n \Delta_k.
	\end{align*}
	Then
	\[
	\Var(T(0, \mathcal{O}_n)) = \sum_{k=0}^n \E\Delta_k^2.
	\]
	Write $\Delta_k^\textnormal{B}$ for the corresponding $\Delta_k$ with Bernoulli weights. (Here, as in the last section, we couple $(t_x)$ with $(t_x^\textnormal{B})$, a family of i.i.d. Bernoulli$(1/2)$ random variables, and write $T^\textnormal{B}$ for the passage time using $(t_x^\textnormal{B})$.) We would like to compare $\Var(T(0, \mathcal{O}_n))$ with $\Var(T^\textnormal{B}(0, \mathcal{O}_n))$, and so we would like to bound $\E \Delta_k^2 - \E(\Delta_k^\textnormal{B})^2$.
	
	First we need another formula for $\Delta_k$. Let $(\Omega',\mathcal{F}',\prob')$ be another copy of the probability space $(\Omega, \mathcal{F}, \prob)$. Let $\E'$ denote the expectation with respect to $\prob'$ and $\omega'$ denote a sample point in $\Omega'$. Define
	\[
	\ell(n, \omega, \omega') = m(m(n,\omega)+1, \omega').
	\]
	Now by \cite[Lemma~2]{KestenZhang}, we have
	\begin{align*}
	\Delta_k (\omega) = &T(\mathcal{O}_{k-1}(\omega), \mathcal{O}_k(\omega))(\omega) + \E' T(\mathcal{O}_k(\omega), \mathcal{O}_{\ell(k, \omega, \omega')}(\omega'))(\omega') \\
	&- \E' T(\mathcal{O}_{k-1}(\omega), \mathcal{O}_{\ell(k, \omega, \omega')}(\omega'))(\omega').
	\end{align*}
	By the Cauchy-Schwarz inequality,
	\begin{align*}
	\left| \mathrm{Var}(T(0,\mathcal{O}_n)) - \mathrm{Var}(T^\textnormal{B}(0,\mathcal{O}_n)) \right|  &\leq \sum_{k=0}^n \left| \E \Delta_k^2 - \E(\Delta_k^\textnormal{B})^2\right| \nonumber \\
	&\leq \sum_{k=0}^n (\E(\Delta_k - \Delta_k^\textnormal{B})^2)^{1/2}(\E(\Delta_k + \Delta_k^\textnormal{B})^2)^{1/2}.
	\end{align*}
	Now,
	\begin{align*}
	\E(\Delta_k + \Delta_k^\textnormal{B})^2 \leq 2 \E \Delta_k^2 + 2 \E\left( \Delta_k^\textnormal{B}\right)^2 &\leq 2\left( \sqrt{\E(\Delta_k - \Delta_k^\textnormal{B})^2} + \sqrt{\E(\Delta_k^\textnormal{B})^2}\right)^2 + 2\E(\Delta_k^\textnormal{B})^2 \\
	&\leq 4\E(\Delta_k - \Delta_k^\textnormal{B})^2 + 6 \E(\Delta_k^\textnormal{B})^2,
	\end{align*}
	so using $\sqrt{a+b} \leq \sqrt{a}+\sqrt{b}$ for $a,b \geq 0$, we obtain
	\begin{align*}
	\left| \mathrm{Var}(T(0,\mathcal{O}_n)) - \mathrm{Var}(T^\textnormal{B}(0,\mathcal{O}_n)) \right| &\leq 2\sum_{k=0}^n \E(\Delta_k - \Delta_k^\textnormal{B})^2 \\
	&+\sqrt{6}\left( \sup_k \sqrt{\E(\Delta_k^\textnormal{B})^2}\right) \sum_{k=0}^n (\E(\Delta_k-\Delta_k^\textnormal{B})^2)^{1/2}.
	\end{align*}
	
	Due to \cite[Lemma~5.5]{critical}, $\sup_k \E(\Delta_k^\textnormal{B})^2 < \infty$, and so
	\begin{equation}\label{eq: variance_difference}
	\left| \mathrm{Var}(T(0,\mathcal{O}_n)) - \mathrm{Var}(T^\textnormal{B}(0,\mathcal{O}_n)) \right| \leq C_{8} \sum_{k=0}^n \left[ (\E(\Delta_k - \Delta_k^\textnormal{B})^2)^{1/2} + \E(\Delta_k - \Delta_k^\textnormal{B})^2\right].
	\end{equation}
	
	To bound $\mathbf{E}(\Delta_k - \Delta_k^\textnormal{B})^2$, we write
	\[
	X = X(k) = T(\mathcal{O}_{k-1}(\omega), \mathcal{O}_k(\omega))(\omega) - T^\textnormal{B}(\mathcal{O}_{k-1}(\omega), \mathcal{O}_k(\omega))(\omega),
	\]
	\[
	Y = Y(k) = T(\mathcal{O}_k(\omega), \mathcal{O}_{\ell(k, \omega, \omega')}(\omega'))(\omega') - T^\textnormal{B}(\mathcal{O}_k(\omega), \mathcal{O}_{\ell(k, \omega, \omega')}(\omega'))(\omega'),
	\]
	and
	\[
	Z = Z(k) = T(\mathcal{O}_{k-1}(\omega), \mathcal{O}_{\ell(k, \omega, \omega')}(\omega'))(\omega') - T^\textnormal{B}(\mathcal{O}_{k-1}(\omega), \mathcal{O}_{\ell(k, \omega, \omega')}(\omega'))(\omega').
	\]
	Then by the Cauchy-Schwarz and Jensen inequalities,
	\begin{equation}\label{eq: x_y_z}
	\E(\Delta_k - \Delta_k^\textnormal{B})^2 = \E(X + \E'Y - \E'Z)^2 \leq 9\max\{\E X^2, \E\E' Y^2, \E\E' Z^2\}.
	\end{equation}
	We will only bound $\E\E' Y^2$, as bounding $\E\E' Z^2$ is similar and bounding $\E X^2$ is even easier. 

	We first give an alternate representation for $\E\E' Y^2$ which only depends on $\omega$.
	\begin{lem}\label{lem: y_tilde}
	One has
	\begin{equation}\label{eq: y_tilde}
	\E\E' Y^2 = \E \tilde{Y}^2,
	\end{equation}
	where $\tilde{Y} = T(\mathcal{O}_k(\omega), \mathcal{O}_{\ell(k, \omega, \omega)}(\omega))(\omega) - T^\textnormal{B}(\mathcal{O}_k(\omega), \mathcal{O}_{\ell(k, \omega, \omega)}(\omega))(\omega)$. 
	\end{lem}
	\begin{proof}To show this, it suffices to show that
	\[
	(T_Y, T_Y^\textnormal{B}) = (T_{\tilde{Y}}, T_{\tilde{Y}}^\textnormal{B}) \text{ in distribution},
	\]
	where
	\[
	T_Y = T(\mathcal{O}_k(\omega), \mathcal{O}_{\ell(k, \omega, \omega')}(\omega'))(\omega'),
	\]
	\[
	T_{\tilde{Y}} = T(\mathcal{O}_k(\omega), \mathcal{O}_{\ell(k, \omega, \omega)}(\omega))(\omega)
	\]
	and $T_Y^\textnormal{B}, T_{\tilde Y}^\textnormal{B}$ are defined analogously using the Bernoulli variables $(t_x^\textnormal{B})$. To prove this, note that for any Borel set $E \subseteq \R^2$,
	\begin{align*}
	&\left(\prob\times\prob'\right)((T_Y, T_Y^\textnormal{B}) \in E) \\
	=~&\sum_{\ell \geq k} \sum_{\Gamma\subseteq A(\ell)} \left(\prob\times\prob'\right)((T_Y, T_Y^\textnormal{B}) \in E, m(k,\omega) = \ell,\mathcal{O}_k = \Gamma).
	\end{align*}
	The summand equals
	\begin{align*}
	&\left( \prob\times\prob'\right)((T(\Gamma,\mathcal{O}_{m(\ell+1,\omega')}(\omega'))(\omega'), T^\textnormal{B}(\Gamma,\mathcal{O}_{m(\ell+1,\omega')}(\omega'))(\omega')) \in E, m(k,\omega) = \ell, \mathcal{O}_k = \Gamma) \\
	=~&\prob'((T(\Gamma,\mathcal{O}_{m(\ell+1,\omega')}(\omega'))(\omega'), T^\textnormal{B}(\Gamma,\mathcal{O}_{m(\ell+1,\omega')}(\omega'))(\omega')) \in E) \prob(m(k,\omega) = \ell, \mathcal{O}_k = \Gamma)\\
	=~&\prob((T(\Gamma,\mathcal{O}_{m(\ell+1,\omega)}(\omega))(\omega), T^\textnormal{B}(\Gamma,\mathcal{O}_{m(\ell+1,\omega)}(\omega))(\omega)) \in E) \prob(m(k,\omega) = \ell, \mathcal{O}_k = \Gamma).
	\end{align*}
	Note that the event $\{m(k,\omega) = \ell, \mathcal{O}_k = \Gamma\}$ depends only on variables associated to vertices in $\overline{\Gamma}$. So we can regroup the probabilities by independence, and reverse the steps to obtain $\prob((T_{\tilde{Y}}, T_{\tilde{Y}}^\textnormal{B}) \in E)$. This shows \eqref{eq: y_tilde}.
	\end{proof}
	
	We will also need some preliminary bounds on moments of $\tilde{Y}$ and $\Delta_k$.
	\begin{lem}\label{lem: preliminary_bounds}
	One has $\E \tilde Y^2 < \infty$ for all $k$. Also, there exists $C_{9} <\infty$ and $k_0$ such that if $k \geq k_0$, then $\E \tilde Y^4 \leq C_{9}$. The same bounds hold for $\Delta_k$: one has $\E \Delta_k^2 < \infty$ for all $k$ and $\E \Delta_k^4 \leq C_{9}$ for all $k \geq k_0$.
	\end{lem}
	\begin{proof}
	The proofs for $\Delta_k$ and $\tilde Y$ are very similar, so we show the case of $\tilde Y$. The first step is to (non-optimally) bound the $p$-th moment of annulus passage times. Let $m,n$ be nonnegative integers with $m \leq n$. We will show that for any $p>0$, there exist $C_{10}, C_{11} > 0$ and $m_0$ such that 
	\begin{equation}\label{eq: pasta_supremo}
	\E T(B(m), \partial B(n))^p \leq C_{10} \left( \log \frac{n}{m} \right)^{C_{11}} \text{ if } n \geq m \geq m_0.
	\end{equation}

	To do this, we take $R$ to be a large fixed integer, and construct $R$ disjoint sectors as follows. Define the first sector $S_1$ to be the open region of $\mathbb{R}^2$ whose boundary consists of the circle of radius $m$ centered at 0, the circle of radius $2n$ centered at zero, the positive $e_1$-axis, and the ray started at 0 (in the first quadrant) with angle $\pi/R$ with the positive $e_1$-axis. $S_i$ for $i=2, \dots, R$ is the rotation of $S_1$ by the angle $\pi (i-1) / R$. (We choose the constant $m_0$ to prevent the sectors from being exiguous.) For $i=1, \dots, R$, let $\pi_i$ be a path connecting the inner boundary of $S_i$ to the outer boundary with the minimal number $N_i$ of nonzero-weight vertices. Then since the variables $T(\pi_i)$ are independent, for $\epsilon \in (0,1/6)$ (so that a vertex-weight $t_v$ satisfies $\E t_v^\epsilon<\infty$),
	\begin{align}
	\E T(B(m), \partial B(n))^p &\leq \E \min\{T(\pi_1), \dots, T(\pi_R)\}^p \nonumber \\
	&\leq 1+ \int_1^\infty py^{p-1} \max_i \prob(T(\pi_i) \geq y)^R~\text{d}y \nonumber \\
	&\leq 1+ \int_1^\infty py^{p-1}  \left( \frac{\max_i \E T^\epsilon (\pi_i)}{y^\epsilon}\right)^R~\text{d}y\nonumber \\
	&= 1+  \left( \max_i \E T^{\epsilon}(\pi_i)\right)^R \int_1^\infty py^{p-1-R\epsilon}~\text{d}y \nonumber\\
	&\leq C_{12} \left( 1+ \left( \max_i \E T^{\epsilon}(\pi_i)\right)^R\right) \label{eq: fritatta},
	\end{align}
	so long as $R > p/\epsilon$. The inner expected value is computed by introducing $K = \log^3 (n/m)$ and writing
	\begin{align}
	&\E T^\epsilon (\pi_i) \mathbf{1}_{\{N_i \leq K\}} + \sum_{j=1}^\infty \E T^\epsilon(\pi_i)\mathbf{1}_{\{N_i \in (jK,(j+1)K]\}} \nonumber \\
	\leq~&\E [t_1 + \dots + t_K]^\epsilon + \sum_{j=1}^\infty \left( \E[t_1 + \dots + t_{(j+1)K}]^\epsilon \prob(N_i \geq jK)\right) \nonumber \\
	\leq~&K\E t_1^\epsilon\left( 1 + \sum_{j=1}^\infty (j+1) \prob(N_i \geq jK)\right). \label{eq: fritatta_2}
	\end{align}
	Here, we have written $t_1, t_2, \dots$ for i.i.d.~vertex-weights and used that the variable $N_i$ is independent of the non-zero weights of vertices on $\pi_i$.
	
		\begin{figure}[h]
		\centering
		\includegraphics[width=.6\textwidth]{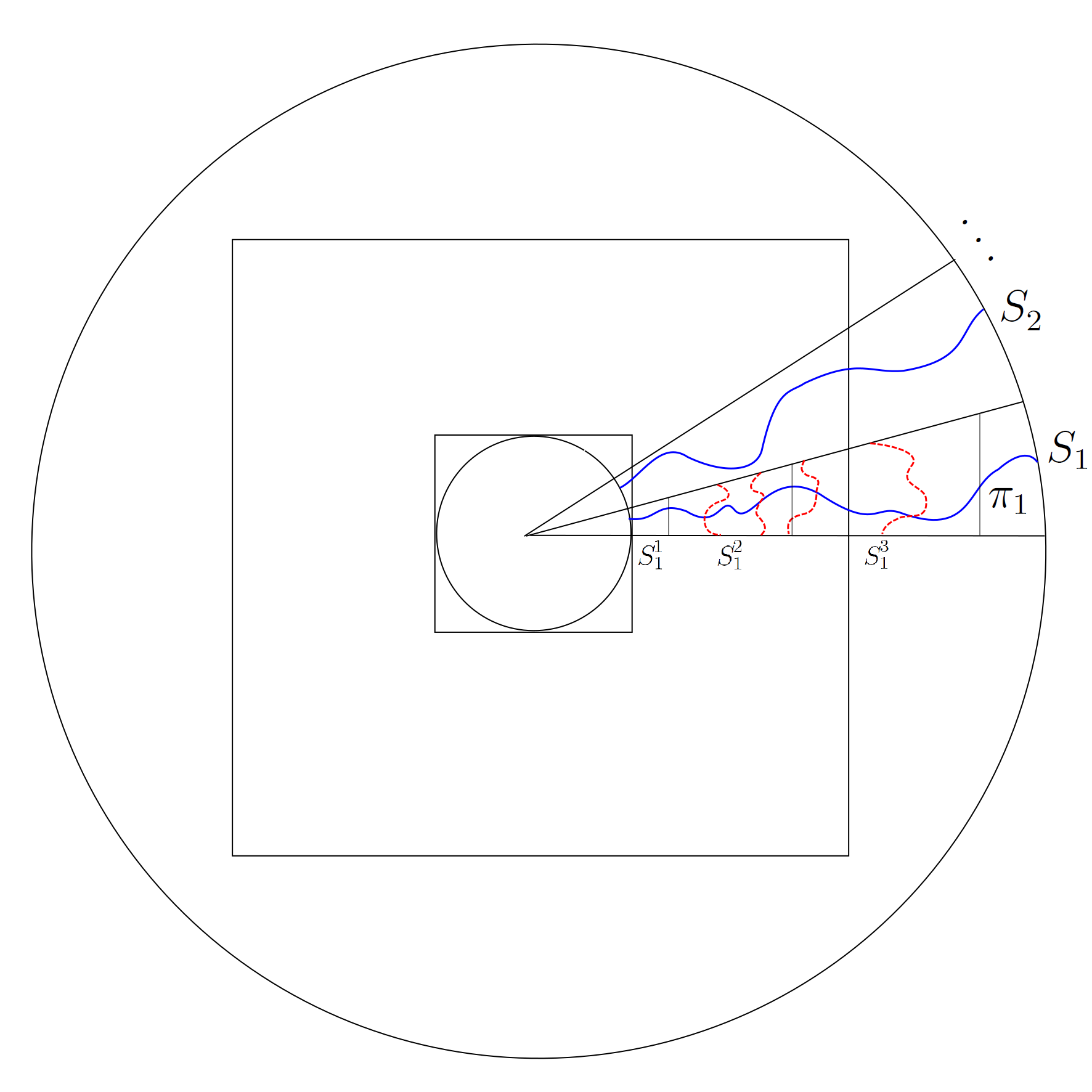}
		\caption{The blue solid lines are the $\pi_i$'s, paths connecting the inner boundary of $S_i$ to the outer boundary with the minimal number $N_i$ of nonzero-weight vertices. On $S_1$ (say), for each nonzero-weight vertex that $\pi_1$ passes through, by minimality and planar duality, there must be a closed path containing that vertex and connecting the bottom of $S_1$ to the top of $S_1$. These closed paths are the red dotted curves in the figure.}
		\label{fig: sectors}
	\end{figure}
	
	Now to bound $\prob(N_i \geq jK)$, we note that if $N_i \geq jK$, then by planar duality, there must be at least $jK$ many disjoint paths connecting the side boundaries of $S_i$ to each other, and consisting only of nonzero-weight vertices. Splitting $S_i$ into at most $\left\lceil \log_2 (n/m)\right\rceil$ many sets of the form $S_i \cap \left[ B(2^k) \setminus B(2^{k-1})\right]$ and writing these sets as $\{S_i^r\}_r$, we see that at least $jK/\left\lceil \log_2 (n/m) \right\rceil$ many of these paths must intersect some $S_i^r$. By the RSW theorem, the probability that there exists at least one such path is bounded above by $1-c_{14}$ for some $c_{14}$ positive, uniformly in $r,m,n$. By the BK inequality, we therefore obtain
	\begin{align*}
	\prob(N_i \geq jK) &\leq \sum_r \prob(\text{at least }jK/\left\lceil \log_2(n/m)\right\rceil \text{ such paths intersect } S_i^r) \\
	&\leq \left\lceil \log_2(n/m)\right\rceil (1-c_{14})^{jK/\left\lceil \log_2(n/m)\right\rceil} \\
	&\leq C_{13}e^{-c_{15}j}.
	\end{align*}
	We plug this estimate back into \eqref{eq: fritatta_2} to obtain
	\[
	\E T^\epsilon(\pi_i) \leq C_{14} \log^3(n/m),
	\]
	and then back in \eqref{eq: fritatta} to obtain
	\[
	\E T(B(m),\partial B(n))^p \leq C_{15} \log^{3R}(n/m).
	\]
	This shows \eqref{eq: pasta_supremo}.
	
	The next step is to extend \eqref{eq: pasta_supremo} in the case when $p=2$ to all $m,n$ with $n \geq m$. We claim that there are numbers $C_{16}, C_{17}, C_{18}$ such that
	\begin{equation}\label{eq: pasta_supremo_2}
	\E T(B(m),\partial B(n))^2 \leq C_{16} + C_{17}\left( \log \frac{n}{m}\right)^{C_{18}} \text{ if } n \geq m \geq 0.
	\end{equation}
	To do this, we note that if $n\geq m \geq m_0$ (where $m_0$ is from \eqref{eq: pasta_supremo}), then the inequality follows from \eqref{eq: pasta_supremo}. Otherwise $m < m_0$ and either $n \leq m_0+100$, say, or $n > m_0+100$. In the first case, we upper bound
	\begin{equation}\label{eq: quest_1}
	\E T(B(m), \partial B(n))^2 \leq \E T(0,\partial B(m_0+100))^2.
	\end{equation}
	In the second case, we let $\pi$ be a geodesic between $B(m_0)$ and $\partial B(n)$ (chosen in some deterministic way) and $\widetilde \pi$ be the portion of $\pi$ from its initial point to its first intersection with $\partial B(m_0+100)$. Then
	\begin{align}
	\E T(B(m), \partial B(n))^2 &\leq 2\E T(0,\widetilde \pi)^2 + 2\E T(B(m_0), \partial B(n))^2 \nonumber \\
	&\leq 2\sum_P \E T(0,P)^2 + 2\E T(B(m_0), \partial B(n))^2, \label{eq: quest_2}
	\end{align}
	where the sum is over all paths $P$ that start at $\partial B(m_0)$ and end at $\partial B(m_0+100)$.	
	
	We bound the terms in \eqref{eq: quest_1} and \eqref{eq: quest_2} similarly. For example, for \eqref{eq: quest_2}, a moment's reflection shows that one can construct six deterministic paths $\rho_1, \dots, \rho_6$ that start at 0 and end at $P$, and are vertex disjoint except for their initial points. The upper bound we obtain is then (using \eqref{eq: pasta_supremo})
	\[
	\E T(B(m), \partial B(n))^2 \leq 2 \sum_P \E \min\{T(\rho_1), \dots, T(\rho_6)\}^2 + C_{17} \left( \log \frac{n}{m} \right)^{C_{18}}.
	\]
	The argument of \cite[Lemma~3.1]{coxdurrett} shows that since we have assumed $\E \min\{t_1, \dots, t_6\}^2 < \infty$, then the first term on the right is bounded by a constant. The bound on \eqref{eq: quest_1} using this method is just a constant, so this implies \eqref{eq: pasta_supremo_2}.
	
%
	
	We now move to showing that for some $k_0$,
	\begin{equation}\label{eq: task_1}
	\E\tilde Y^4 \leq C_{9} \text{ for all } k\geq k_0.
	\end{equation}
	Given \eqref{eq: pasta_supremo}, the claimed inequality will follow forthwith. Indeed, since $T^\textnormal{B} \leq T$, we write the left side as
	\begin{align}
	&\E \left( T(\mathcal{O}_k(\omega), \mathcal{O}_{\ell(k, \omega, \omega)}(\omega))(\omega) - T^\textnormal{B}(\mathcal{O}_k(\omega), \mathcal{O}_{\ell(k, \omega, \omega)}(\omega))(\omega)\right)^4 \nonumber \\
	\leq~&  \E \left( T(\mathcal{O}_k(\omega), \mathcal{O}_{\ell(k, \omega, \omega)}(\omega))(\omega) \right)^4 \nonumber \\
	=~&\sum_{r=k}^\infty \sum_{s=r+1}^\infty \E (T(\mathcal{O}_k(\omega), \mathcal{O}_{\ell(k,\omega,\omega)}(\omega))(\omega))^4 \mathbf{1}_{\{m(k) = r, \ell(k,\omega,\omega) = s\}} \nonumber \\
	\leq~& \sum_{r=k}^\infty \sum_{s=r+1}^\infty \E T(B(2^r), \partial B(2^{s+1}))^4 \mathbf{1}_{\{m(k) = r, \ell(k,\omega,\omega) = s\}}. \label{eq: mr_hernschtadt}
	\end{align}
	By the Cauchy-Schwarz inequality, we obtain the upper bound
	\[
	 \sum_{r=k}^\infty \sum_{s=r+1}^\infty \sqrt{\E T(B(2^r), \partial B(2^{s+1}))^8 \prob(m(k)=r,\ell(k,\omega,\omega)=s)}.
	\]
	Assuming that $k \geq k_0$, where $k_0$ is large enough so that $B(2^{k_0}) \supseteq B(m_0)$ (here $m_0$ is from \eqref{eq: pasta_supremo} with $p=8$), we can use \eqref{eq: pasta_supremo} to further bound this by
	\begin{align}
	&C_{18} \sum_{r=k}^\infty \sum_{s=r+1}^\infty (s+1-r)^{C_{19}} \sqrt{ \prob(m(k)=r,\ell(k,\omega,\omega)=s)} \label{eq: mr_kit!} \\
	\leq~& C_{19} \sum_{r=k}^\infty e^{-c_{16} (r-k)} \sum_{s=r+1}^\infty (s+1-r)^{C_{20}} e^{-c_{17}(s-r)} \nonumber \\
	\leq~& C_{21} \sum_{r=k}^\infty e^{-c_{18} (r-k)} \nonumber \\
	\leq~& C_{22}. \nonumber
	\end{align}
	In the second line, we have used the RSW theorem. This proves \eqref{eq: task_1}.
	
	Last, to show $\E \tilde Y^2 < \infty$ for all $k$, due to \eqref{eq: task_1}, we need only consider the case when $k < k_0$. Then we move to \eqref{eq: mr_hernschtadt} with an exponent 2 instead of 4:
	\[
	\E \tilde Y^2 \leq \sum_{r=k}^\infty \sum_{s=r+1}^\infty \E T(B(2^r),\partial B(2^{s+1}))^2 \mathbf{1}_{\{m(k)=r,\ell(k,\omega,\omega)=s\}}.
	\]
	The proof from here follows similar lines to that of the above, so we only briefly indicate the idea. For values of $r,s$ such that $r \leq k_0$ and $s \leq k_0+100$, we upper bound by removing the indicator and summing over these (finitely many) values with the bound \eqref{eq: pasta_supremo_2} to obtain a finite number. For values of $r,s$ which are $> k_0$, we apply the Cauchy Schwarz inequality to obtain a sum over such $r,s$ of
	\[
	\sqrt{\E T(B(2^r), \partial B(2^{s+1}))^4 \prob(m(k)=r,\ell(k,\omega,\omega)=s)},
	\]
	and sum this as in the case of bounding $\E \tilde Y^4$. Last, if $r \leq k_0$ but $s > k_0+100$, we bound $T(B(2^r), \partial B(2^{s+1}))$ above by the sum of $T(0,\tilde{\pi})$ and $T(B(2^{k_0}), \partial B(2^{s+1}))$ (where $\tilde{\pi}$ is a geodesic connecting $B(2^{k_0})$ to $\partial B(2^{k_0+100})$ chosen analogously to that above \eqref{eq: quest_2}), and note that the first term has finite second moment. The second term is bounded as in the case where $r,s > k_0$. Combining the cases will produce the final inequality, $\E \tilde Y^2 < \infty$.
	\end{proof}
	
	To bound $\E \tilde{Y}^2$ from \eqref{eq: y_tilde} more tightly, we introduce two events. Choose $c_1,c_3$ and $c_1',c_3'$ with corresponding $c_2=c_2'$ as dictated by Lemma~\ref{lem:good}. Also fix $c \in (c_1,c_1'c_3')$, and $\hat c \in (0,c_1c_3)$. Write $\tilde{\mathcal{C}}_k$ for the first circuit in the sequence in Lemma~\ref{lem: outermost} with $\mathcal{O}_k$ in the interior of $\tilde{\mathcal{C}}_k$, and let
	\[
	F_k = \{\exists\; v\in \mathcal{O}_k \text{ such that } d(v, \tilde{\mathcal{C}}_k) < 2(\diam(\tilde{\mathcal{C}}_k))^{c_1'}\}.
	\]
	Also, for $C_6, C_7, c_{7}>0$, let $G_k$ be the event that at least one of the following fails:
	\begin{itemize}
	\item each circuit from the sequence in Lemma~\ref{lem: outermost} with diameter at least $2^k$ is $(c_1,c_2,c_3)$-good and $(c_1',c_2',c_3')$-good, 
	\item $N(j,c)\leq  j^2$, $N(j,\hat c)\leq j^2$, and $N^{(3)}(j,c_3')= 0$ for all $j \geq k$, 
	\item $\diam(\mathcal{C}_{i+2}) \leq (i+1)^{C_6} \diam(\mathcal{C}_{i+1})$ and $c_7i \leq \log(\diam(\mathcal{C}_{i+2})) \leq C_7i$ for all $i$ such that $\mathcal{C}_i$ has diameter at least $2^k$. 
	\end{itemize}
	Then,
	\begin{equation}\label{eq: y_tilde_breakdown}
	\E\E' Y^2 = \E \tilde{Y}^2 = \E \tilde{Y}^2\mathbf{1}_{F_k^c \cap G_k^c} + \E\tilde{Y}^2 \mathbf{1}_{F_k\cup G_k}.
	\end{equation}
	\begin{lem}\label{lem: G_k_lemma}
	For $C_6$ and $C_{7}$ chosen to be large enough and $c_{7}$ chosen to be small enough, there exist $C_{23}, C_{24}, c_{19}>0$ such that
	\[
	\mathbf{P}(G_k) \leq C_{23}/k
	\]
	and
		\begin{equation}
	\label{eq: base_case}
	\prob(F_k) \leq C_{24}e^{-c_{19}k}.
	\end{equation}
	\end{lem}
	\begin{proof}
	The bound on $\prob(G_k)$ follows from Lemma~\ref{lem:good}, \eqref{eq: Njc_prob}, \eqref{eq: N3_prob}, and the fact that for all $k\geq 1$, 
	\begin{equation}
	\label{eq: fact}
	\prob(\log(\diam(\mathcal{C}_{j+2})) < c_{7}j \text{ or } \geq C_{7}j\textnormal{ for some $j\geq k$})\leq \frac{C_{25}}{k},
	\end{equation}
	and
	\begin{equation}
	\label{eq: fact2}
	\prob(\diam(\mathcal{C}_{j+2}) > (j+1)^{C_6} \diam(\mathcal{C}_{j+1})\textnormal{ for some $j\geq k$})\leq \frac{C_{26}}{k},
	\end{equation}
	Here $C_6, C_7$ are chosen large enough and $c_{7}$ is chosen small enough. One can show both \eqref{eq: fact} and \eqref{eq: fact2} hold for all $k$ by following the proof of \cite[Lemma~4]{KSZ}. We omit the details here.
	
	We now move to the proof of the second statement, the bound on $\mathbf{P}(F_k)$. In $A(r)$ ($r\geq 0$), let $\hat{\mathcal{O}}_r$ be the innermost open circuit (if there exists one) and let $\hat{\mathcal{C}}_r$ be the first circuit in the sequence in Lemma~\ref{lem: outermost} with $\hat{\mathcal{O}}_r$ in the interior of $\hat{\mathcal{C}}_r$. Because $\mathcal{O}_k\subseteq A(m(k))$, the probability $\prob(F_k)$ equals
	\begin{align}
	&\sum_{\ell = k}^{\infty} \prob(\exists\; v\in \mathcal{O}_k \text{ with } d(v, \tilde{\mathcal{C}}_k) < 2(\diam(\tilde{\mathcal{C}}_k))^{c_1'}\mid m(k) = \ell)\prob(m(k) = \ell) \nonumber \\
	\leq~& \sum_{\ell = k}^{\infty} \prob(\exists\; v\in\hat{ \mathcal{O}}_\ell \text{ with } d(v,\hat{\mathcal{C}}_\ell) <2(\diam(\hat{\mathcal{C}}_\ell))^{c_1'} \mid \exists\text{ open circuit around $0$ in $A(\ell)$}) e^{-c_{20}(\ell-k)} \label{eq: label_it!}.
	\end{align} 
	To change the conditioning above, we used independence of the site variables in disjoint annuli, and to bound $\mathbf{P}(m(k) = \ell)$, we used the RSW theorem (see \cite[Eq.~(2.28)]{KestenZhang}).
	
	To bound \eqref{eq: label_it!}, we first show that for any choice of $\nu \in (c_1',1)$, one has $2(\diam(\hat{\mathcal{C}}_\ell))^{c_1'}< 2^{\nu\ell}$ with high probability. To see this, fix $\alpha>1$ and consider $\prob(\diam(\hat{\mathcal{C}}_\ell) \geq 2^{\alpha\ell})$. For $r\geq 0$, define $E_r$ to be the event that
	\begin{itemize}
		\item there exists a closed circuit surrounding $0$ in $A(r)$, and
		\item there exists an open circuit surrounding the above closed circuit in $A(r)$.
	\end{itemize}
	By the RSW theorem, there exists $\kappa>0$ such that $\prob(E_r) > \kappa$ for all $r\geq 0$. Because there is no closed circuit surrounding 0 strictly between $\hat{\mathcal{O}}_\ell$ and $\hat{\mathcal{C}}_\ell$, the event $\{\diam(\hat{\mathcal{C}}_\ell) \leq 2^d\}$ contains $E_d$. Therefore, by independence,
	\begin{equation}\label{eq: alpha_inequality}
	\prob(\diam(\hat{\mathcal{C}}_\ell) \geq 2^{\alpha\ell} \mid  \exists\text{ open circuit around $0$ in $A(\ell)$} ) \leq \prob\left(\bigcap_{k=1}^{p(\ell)} E_{\ell+k}^c\right) \leq e^{-\beta p(\ell)}
	\end{equation}
	for some $\beta>0$ depending only on $\alpha$, and $p(\ell)$ is the integer such that
	\[
	\diam(B(2^{\ell+p(\ell)})) \leq 2^{\alpha \ell} < \diam(B(2^{\ell+p(\ell)+1})).
	\]
	Solving for $p(\ell)$, we see that $p(\ell) > (\alpha-1)\ell - 5/2$. Thus 
	\[
	\prob(\diam(\hat{\mathcal{C}}_\ell) \geq 2^{\alpha\ell}) \leq e^{-\beta' \ell}
	\]
	for some $\beta'>0$ depending only on $\alpha$.
	
	In particular, for any $\nu\in(c_1',1)$, we can apply \eqref{eq: alpha_inequality} with $\alpha = \frac{1+\nu/c_1'}{2}$ to obtain 
	\begin{align}
	&\prob(\exists\; v\in\hat{ \mathcal{O}}_\ell \text{ such that } d(v,\hat{\mathcal{C}}_\ell) <2(\diam(\hat{\mathcal{C}}_\ell))^{c_1'} \mid \exists\text{ open circuit around $0$ in $A(\ell)$}) \nonumber\\
	\leq~& \prob(\exists\; v\in\hat{ \mathcal{O}}_\ell \text{ such that } d(v,\hat{\mathcal{C}}_\ell) <2^{\nu\ell} \mid \exists\text{ open circuit around $0$ in $A(\ell)$}) + e^{-\beta'' \ell}  \nonumber \\
	\leq~&C_{27}\prob(\exists\; v\in\hat{ \mathcal{O}}_\ell \text{ such that } d(v,\hat{\mathcal{C}}_\ell) <2^{\nu\ell}) + e^{-\beta''\ell} \label{eq: beta_double_prime}.
	\end{align}
	for some $\beta''>0$ depending on $\nu$ and $c_1'$. Here, we have used the RSW theorem to remove the conditioning.

	We now cover $A(\ell)$ by the squares $\tau(r,s)$ from Lemma~\ref{lem:squares}. These were defined as
	\[
	\tau(r,s) = \tau(r,s;\nu,\ell) = \left[ r2^{\nu(\ell+1)}, (r+3)2^{\nu(\ell+1)}\right] \times \left[ s2^{\nu(\ell+1)},(s+3)2^{\nu(\ell+1)}\right],
	\]
	where $\nu$ (as above) is any number in $(c_1',1)$. These squares are defined so that any $v \in A(\ell)$ is in the central square $\tau'(r,s)$ of sidelength $2^{\nu(\ell+1)}$ of some $\tau(r,s)$.
	
	To bound the probability in \eqref{eq: beta_double_prime}, we need the notion of arm events. For any $j\geq 1$, any sequence $\sigma = (\sigma_1,\ldots,\sigma_j)$, where $\sigma_i\in\{\text{open}, \text{closed}\}$, and any positive integers $n,N$ with $n\leq N$, we define the $j$-arm event, $A_{j, \sigma}(n, N)$, to be the event that there exist $j$ disjoint paths from $B(n)$ to $\partial B(N)$, and the $i$-th path has occupation status $\sigma_i$ (either open or closed), taken in counterclockwise order. Similarly, we define $A^{1/2}_{j, \sigma}(n, N)$ and $A^{3/4}_{j, \sigma}(n, N)$ in the same way, but further restrict the paths to lie entirely in the upper half plane and in the union of the first three quadrants respectively. It is known (see for instance \cite{Nolin} or \cite{Smirnov}) that if the $\sigma_i$'s are not all open or all closed (this is irrelevant except in the full plane),
	\begin{equation}
	\label{eq: exponents}
	\prob(A_{j, \sigma}(n, N)) = \left(\frac{N}{n}\right)^{-(j^2-1)/12 + o(1)} ,\quad \prob(A^{1/2}_{j, \sigma}(n, N)) = \left(\frac{N}{n}\right)^{-j(j+1)/6 + o(1)} \text{ as } N/n \to \infty.
	\end{equation}
	Further, by conformal invariance of limiting crossing probabilities, one can also deduce that
	\begin{equation}
	\label{eq: 3/4_exponent}
	\prob(A^{3/4}_{j, \sigma}(n, N)) = \left( \frac{N}{n}\right)^{-j(j+1)/9 + o(1)} \text{ as } N/n \to \infty.
	\end{equation}
	
	For our chosen $\nu \in (c_1',1)$, now pick $\nu' \in (\nu, 1)$, and suppose that the event 
	\[
	\{\exists\; v\in\hat{ \mathcal{O}}_\ell \text{ such that } d(v,\hat{\mathcal{C}}_\ell) <2^{\nu\ell}\}
	\] 
	occurs and select such a $v$. We will sketch the idea of how arm events help us to bound the probability of this event. Because it is a standard ``arms reckoning'' argument, we omit details and refer the reader to similar arguments in \cite[Lemma~6]{KSZ}. Choose a central square $\tau'(r,s)$ that contains $v$. If the distance between $\tau'(r,s)$ and the box $B(2^\ell)$ is at least $2^{\nu'\ell}$, then there is a 6-arm event in $B(v, 2^{\nu'\ell}) \setminus \tau'(r,s)$ (here $B(v,q) = v+B(q)$). Otherwise, it will yield a 5-arm event on a half plane or on a $3/4$-plane, depending on the position of $\tau'(r,s)$.
	
	\begin{figure}[h]
		\centering
		\includegraphics[width=.6\textwidth]{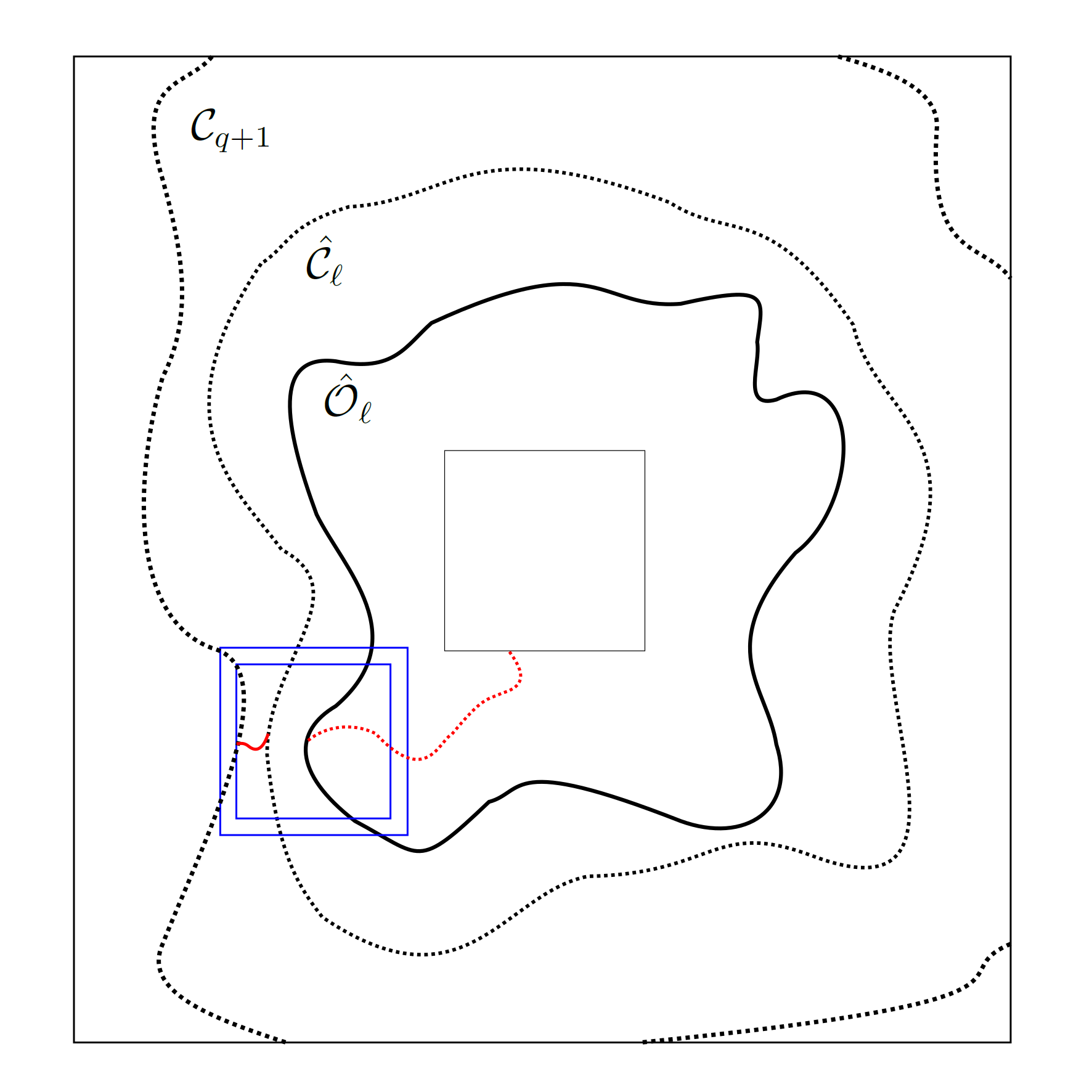}
		\caption{If $\hat{\mathcal{O}}_\ell$ is too close to $\hat{\mathcal{C}}_\ell = \mathcal{C}_q$, then locally there is a $6$-arm event (see the blue annulus). Four of the six arms are furnished by the circuits $\hat{\mathcal{O}}_\ell$ and $\hat{\mathcal{C}}_\ell$ themselves, and the other two, drawn in red, exist because $\hat{\mathcal{O}}_\ell$ is innermost and $\hat{\mathcal{C}}_\ell$ is outermost in the interior of $\mathcal{C}_{q+1}$. If $\mathcal{C}_{q+1}$ is also close to $\hat{\mathcal{C}}_\ell$, there is even a $7$-arm event.}
		\label{fig: armevents}
	\end{figure}
	
	The reason that a 6-arm event will occur is as follows. First, since $\hat{ \mathcal{O}}_\ell$ is the innermost open circuit in $A(\ell)$, for any vertex in $\hat{ \mathcal{O}}_\ell$, in particular $v$, there exists a closed path from that vertex to the inner boundary of $A(\ell)$. Since $B(v, 2^{\nu'\ell}) \subseteq A(\ell)$, there is a closed path from $\tau'(r,s)$ to $\partial B(v,2^{\nu'\ell})$. The circuit $\hat{\mathcal{O}}_\ell$ furnishes two more disjoint open paths. Similarly, because $\hat{\mathcal{C}}_\ell = \mathcal{C}_q$ is a circuit in the sequence of Lemma~\ref{lem: outermost}, for any vertex on $\hat{\mathcal{C}}_\ell$, there exists an open path from that vertex to the next circuit in the sequence, $\mathcal{C}_{q+1}$. If $\mathcal{C}_{q+1}$ does not intersect $B(v, 2^{\nu'\ell}) \setminus \tau'(r,s)$, then we choose for our other three paths the open path mentioned in the previous sentence, and two disjoint closed paths furnished by the circuit $\hat{\mathcal{C}}_\ell$. These would give the final three of the required six arms.
	
	If, on the other hand, $\mathcal{C}_{q+1}$ does intersect $B(v,2^{\nu'\ell}) \setminus \tau'(r,s)$, then we have six arms from $\tau'(r,s)$ to the boundary of some intermediate annulus (between $\tau'(r,s)$ and $B(v,2^{\nu'\ell})$), and then 7 arms from the boundary of this annulus to $\partial B(v,2^{\nu'\ell})$ (see Figure~\ref{fig: armevents}). (The circuits $\hat{\mathcal{C}}_\ell$ and $\mathcal{C}_{q+1}$ furnish four more arms.) Writing $\pi_j(n_1, n_2)$ for the probability of the $j$-arm event in $B(2^{\nu\ell+(\nu'-\nu)n_2}) \setminus B(2^{\nu\ell+(\nu'-\nu)n_1})$, where $n_2>n_1$, and summing over possible positions of the intermediate annulus, we have the following probability bound corresponding to cases in which $\tau'(r,s)$ has distance at least $2^{\nu'\ell}$ from $B(2^\ell)$:
	\[
	C_{28}\#\{\tau'(r,s) \text{ intersecting } B(2^{\ell+1})\} \sum_{r=1}^{\ell-1} \pi_6(0, r)\pi_7(r+1,\ell).
	\]
	The sum is further bounded above by
	\[
	C_{29} 2^{(2-2\nu)\ell} \sum_{r=1}^{\ell - 1} 2^{-(35/12+o(1))(\nu'-\nu)r}2^{-(4+o(1))(\nu'-\nu)(\ell-r)} \leq C_{30}2^{(-(35/12+o(1))(\nu'-\nu) + (2-2\nu))\ell},
	\]
	where the exponents $-35/12$ and $-4$ come from \eqref{eq: exponents} (putting $j=6$ and $7$ respectively). 
	The right side goes to $0$ if $\nu' \in \left(\frac{24}{35}+\frac{11}{35}\nu, 1\right)$. 
%
	
	To deal with the cases where $\tau'(r,s)$ is close to either a corner or a side of $B(2^{\ell})$, we use a similar argument, but further decomposing the arm events according to their distance to $B(2^\ell)$. Because the exponents for the 5-arm event in the half plane and in the $3/4$-plane are $5$ and $10/3$ respectively (putting $j=5$ in \eqref{eq: exponents} and \eqref{eq: 3/4_exponent}), and both are greater than $2$, we are able to choose $\nu'$ close enough to 1 so that the probability corresponding to such $\tau'(r,s)$ is also small. Putting together all the cases, we have
	\[
	\prob(\exists\; v\in\hat{ \mathcal{O}}_\ell \text{ such that } d(v,\hat{\mathcal{C}}_\ell) <2^{\nu\ell}) \leq C_{31} e^{-c_{21}\ell}.
	\]
	Plugging this back into \eqref{eq: beta_double_prime}, and then back into \eqref{eq: label_it!} completes the proof.
	\end{proof}

	We return to \eqref{eq: y_tilde_breakdown}. First, the inequality
	\begin{equation}\label{eq: first_term}
	\E\tilde{Y}^2 \mathbf{1}_{F_k\cup G_k} \leq C_{32}/\sqrt{k}.
	\end{equation}
	follows directly from Lemmas~\ref{lem: preliminary_bounds} and \ref{lem: G_k_lemma}. (For small $k$, we remove the indicator and use Lemma~\ref{lem: preliminary_bounds} only, and for larger $k$, we bound by $\sqrt{\E \tilde Y^4 \prob(F_k \cup G_k)}$ and apply both lemmas.)
%

	Next, we bound the second term of \eqref{eq: y_tilde_breakdown} by showing that for large $k$,
	\begin{equation}\label{eq: second_term}
	\E \tilde{Y}^2\mathbf{1}_{F_k^c \cap G_k^c} \leq C_{33}a_k^2 + C_{34}e^{-c_{8}k/2}.
	\end{equation}
	To do this, we apply Remark~\ref{rem: construction}. Since $\tilde{\mathcal{C}_k}$ is the first circuit in the sequence from Lemma~\ref{lem: outermost} outside of $\mathcal{O}_k$, there are no closed circuits in the region strictly between them. By planar duality, there must be an open path $\mathcal{D}$ connecting a neighbor of $\mathcal{O}_k$ to a neighbor of $\tilde{\mathcal{C}_k}$. Choose $v^{(k)}$ to be any vertex of $\mathcal{O}_k$ adjacent to this open path. Because the event $F_k^c$ occurs, $v^{(k)}$ must have distance at least $2(\diam(\tilde{\mathcal{C}_k}))^{c_1'}$ from $\tilde{\mathcal{C}_k}$, and so the diameter of the open path $\mathcal{D}$ is at least $(\diam(\tilde{\mathcal{C}_k}))^{c_1'}$. Together with the conditions comprising the event $G_k^c$, this is sufficient to invoke the remark, and to deduce that, with conditional probability $\geq 1-e^{-c_{8}k}$ (conditioned on $\eta$), there are sequences $(v^{(i)})_{i \geq k+1}$ and $(\mathcal{D}^{(i)})_{i \geq k}$ as in the conclusion of Lemma~\ref{lem: induction}. In particular, we may find an infinite path $\gamma$ starting from a neighbor of $v^{(k)}$ which consists of only zero-weight vertices or low-weight vertices which are on the circuits $\mathcal{C}_j$ (at most one from each $\mathcal{C}_j$) from Lemma~\ref{lem: outermost}.	Letting $\Upsilon$ be the event that this $\gamma$ exists, we obtain by the Cauchy-Schwarz inequality that for large $k$,
	\begin{align}
	\E \tilde{Y}^2 \mathbf{1}_{F_k^c \cap G_k^c} &\leq \E \tilde{Y}^2 \mathbf{1}_{\Upsilon} + \E\left[ \E[ \tilde{Y}^2 \mathbf{1}_{\Upsilon^c} \mid \eta] \mathbf{1}_{F_k^c \cap G_k^c}\right] \nonumber \\
	&\leq \E \tilde{Y}^2 \mathbf{1}_\Upsilon + e^{-c_{8}k/2} \E\left[ \sqrt{ \E[ \tilde{Y}^4 \mid \eta]} \mathbf{1}_{F_k^c \cap G_k^c}\right] \nonumber \\
	&\leq \E \tilde{Y}^2 \mathbf{1}_\Upsilon + C_{34}e^{-c_{8}k/2}. \label{eq: priliminri_buundz}
	\end{align}
	where we used Lemma~\ref{lem: preliminary_bounds} in the last line.
	
	On the event $\Upsilon$, write $\gamma_k$ for the segment of $\gamma$ beginning at $v^{(k)}$ and ending at the first intersection of $\gamma$ with $\mathcal{O}_{\ell(k,\omega,\omega)}$. Then
	\[
	T^\textnormal{B}(\mathcal{O}_k,\mathcal{O}_{\ell(k,\omega,\omega)}) \leq T(\mathcal{O}_k,\mathcal{O}_{\ell(k,\omega,\omega)}) \leq T(\gamma_k).
	\]
	All vertices $w$ on $\gamma_k$ which have nonzero weight are of low weight, and so such $w$ satisfy $t_w \leq I+a_k$. (Here we use that the sequence $(a_j)$ is nonincreasing.) Distinct $w$'s on $\gamma_k$ correspond to distinct circuits $\mathcal{C}_j$. Therefore if we define
	\[
	N_L = \text{ maximal number of disjoint closed circuits around } 0 \text{ intersecting $B(2^L) \setminus B(2^k)$},
	\]
	then we have
	\begin{equation}\label{eq: come_back_here!}
	\E \tilde{Y}^2\mathbf{1}_{\Upsilon} \leq a_k^2 \E N_{\ell(k,\omega,\omega)}^2.
	\end{equation}
	Next, we use \cite[Lemma~2]{YaoAsymptotics}, which, in our context, states that $\E N_L^4 \leq C_{35}\log^4(2^L/2^k)$, and this is bounded by $C_{36}(L-k)^4$. Therefore, the expectation in \eqref{eq: come_back_here!} equals
	\begin{align*}
	\E N_{\ell(k,\omega,\omega)}^2 &= \sum_{m=k}^\infty \sum_{t = 0}^{\infty} \E N_{\ell(k,\omega,\omega)}^2 \mathbf{1}_{\{m(k,\omega) = m\}} \mathbf{1}_{\{\ell(k,\omega,\omega) = m+t+1\}}\\
	&\leq \sum_{m=k}^\infty \sum_{t = 0}^{\infty} (\E N_{m+t+1}^4)^{1/2} \prob(m(k,\omega) = m, \ell(k,\omega,\omega) = m+t+1)^{1/2}\\
	&\leq \sum_{m=k}^\infty \sum_{t = 0}^{\infty} C_{36}^{1/2} (1+t+m-k)^2 \prob(m(k,\omega) = m, \ell(k,\omega,\omega) = m+t+1)^{1/2}.
	\end{align*}
	Since $\prob(m(k,\omega) = m, \ell(k,\omega,\omega) = m+t+1) \leq C_{37}e^{-c_{22}(t+m-k)}$ by the RSW theorem, the above expression is summable and independent of $k$. Returning to \eqref{eq: come_back_here!}, and placing this in \eqref{eq: priliminri_buundz}, we obtain $\E \tilde{Y}^2\mathbf{1}_{F_k^c\cap G_k^c} \leq C_{33}a_k^2 + C_{34}e^{-c_{8}k/2}$. This shows \eqref{eq: second_term}.
	
	Together, \eqref{eq: first_term} and \eqref{eq: second_term} show that $\E \tilde{Y}^2 \leq C_{38}\max\{a_k^2, k^{-1/2}\}$. $X$ and $Z$ from \eqref{eq: x_y_z} can be bounded similarly, so returning to \eqref{eq: variance_difference} yields
	\begin{equation}\label{eq: smores_quest_bar}
	\left| \Var(T(0, \mathcal{O}_n)) - \Var(T^\textnormal{B}(0, \mathcal{O}_n))\right| \leq C_{39} \sum_{k=1}^n \max\{a_k^2+a_k, k^{-1/2}+k^{-1/4}\} = o(n).
	\end{equation}
	
	Finally, we argue that the above inequality implies
	\begin{equation}
	\label{eq: var_diff}
	\left| \Var(T(0, \partial B(n))) - \Var(T^\textnormal{B}(0, \partial B(n)))\right| = o(\log{n}),
	\end{equation}
	which, along with Yao's results quoted above Theorem~\ref{thm:lln}, gives \eqref{eq: var}. The main ingredient in the proof of \eqref{eq: var_diff} is the following moment bound. (Our argument is modified from \cite[Lemma~5.7]{critical}). There exists $C_{38}$ such that for all sufficiently large $n,q \geq 1$ such that $2^{q-1} \leq n < 2^q$,
	\begin{equation}\label{eq: strawberry_cheesecake_quest_bar}
	\E(T(0,\partial B(n)) - T(0,\mathcal{O}_q))^2 \leq C_{40}.
	\end{equation}
	The same method can be used (or \cite[Lemma~5.7]{critical} can be used directly) to show the corresponding statement for $T^\textnormal{B}$ in place of $T$. Observe that for $1 \leq \ell \leq q$, on the event $\{m(q-\ell) \geq q-1 > m(q-\ell-1)\} \cap \{m(q)=q+t\}$, we have
	\[
	|T(0,\partial B(n)) - T(0,\mathcal{O}_q)| \leq T(\partial B(2^{q-\ell-1}), \partial B(2^{q+t+1}))
	\]
	and on the event $\{m(q-\ell)\geq q-1 > m(q-\ell-1)\}$, we have
	\[
	|T(0,\partial B(n)) - T(0,\mathcal{O}_q)| \leq T(\partial B(2^{q-\ell-1}), \mathcal{O}_q).
	\]

	Then define the events $A_\ell := \{m(q-\ell)\geq q-1 > m(q-\ell-1)\}$, for $1 \leq \ell \leq q$, and $B_t:= \{m(q)=q+t\}$, for $t \geq 0$. Using the above inequalities and the fact that the events $A_\ell \cap B_t$ (over all $t,\ell$) cover the whole probability space, we have
	\begin{align}
	\E(T(0,\partial B(n))- T(0,\mathcal{O}_q))^2 &\leq \sum_{\ell=1}^{q-q_0} \sum_{t=0}^\infty \E T^2(\partial B(2^{q-\ell-1}), \partial B(2^{q+t+1})) \mathbf{1}_{A_\ell \cap B_t} \nonumber\\
	&+\sum_{\ell=q-q_0+1}^q \E T^2(\partial B(2^{q-\ell-1}), \mathcal{O}_q) \mathbf{1}_{A_\ell}\label{eq: quest_3} .
	\end{align}
	Here, $q_0$ is equal to $\lceil \log_2 m_0\rceil +1$, where $m_0$ is from \eqref{eq: pasta_supremo}. For summands in the first line, we use the Cauchy-Schwarz inequality to bound them by
	\[
	\sqrt{\E T^4(\partial B(2^{q-\ell-1}), \partial B(2^{q+t+1}))} \left(\prob(A_\ell)\prob(  B_t)\right)^{1/4} \leq C_{41} (t+\ell+2)^{C_{42}} \left( \prob(A_\ell) \prob( B_t)\right)^{1/4}.
	\]
	For summands in the second line, assuming that $2^q \geq m_0+100$, we replicate the argument leading to \eqref{eq: quest_2}. Specifically, letting $\pi$ be a geodesic between $B(m_0)$ and $\mathcal{O}_q$, and $\tilde \pi$ be the portion of $\pi$ from its initial point to its first intersection with $\partial B(m_0+100)$, then the summand of \eqref{eq: quest_3} is at most
	\begin{equation}\label{eq: quest_4}
	2\E T^2(0,\tilde \pi) + 2\sum_{t=0}^\infty \E T^2(B(m_0), \partial B(2^{q+t+1}))\mathbf{1}_{A_\ell \cap B_t}.
	\end{equation}
	We can then, as before, sum over all possible values of $\tilde \pi = P$, and bound the passage time from 0 to $P$ using six disjoint (except for their initial points) deterministic paths. This leads to the bound $\E T^2(0,\tilde \pi) \leq C_{41}$. Applying the Cauchy-Schwarz inequality to the other term gives the following bound for \eqref{eq: quest_4}:
	\[
	C_{43} + C_{41} \sum_{t=0}^\infty (q+t+1-\log_2 m_0)^{C_{42}} \left( \prob(A_\ell) \prob( B_t)\right)^{1/4}.
	\]
	Using these inequalities in \eqref{eq: quest_3} gives
	\begin{align*}
	\E (T(0,\partial B(n)) - T(0, \mathcal{O}_q))^2 &\leq \sum_{\ell=1}^{q-q_0} \sum_{t=0}^\infty C_{39} (t+\ell+2)^{C_{42}} \left( \prob(A_\ell) \prob( B_t)\right)^{1/4} \\
	&+q_0 C_{43} \\
	&+ C_{41} \sum_{\ell=q-q_0+1}^q \sum_{t=0}^\infty  (q+t+1-\log m_0)^{C_{42}} \left( \prob(A_\ell) \prob( B_t)\right)^{1/4}.
	\end{align*}
	By the RSW theorem, there is $c_{23}$ such that $\prob(A_\ell) \leq e^{-c_{23}\ell}$ and $\prob(B_t) \leq e^{-c_{23}t}$, so this leads to \eqref{eq: strawberry_cheesecake_quest_bar}.
	
	Now that we have proved \eqref{eq: strawberry_cheesecake_quest_bar}, we can quickly derive \eqref{eq: var_diff}. Indeed, we estimate as follows, with $2^{q-1} \leq n < 2^q$:
	\begin{align}
	\left| \Var(T(0, \partial B(n))) - \Var(T^\textnormal{B}(0, \partial B(n)))\right| &\leq \left| \Var(T(0, \partial B(n))) - \Var(T(0,\mathcal{O}_q))\right| \label{eq: end_1}\\
	&+ \left| \Var(T^\textnormal{B}(0,\partial B(n))) - \Var(T^\textnormal{B}(0,\mathcal{O}_q))\right| \label{eq: end_2}\\
	&+ \left| \Var(T(0,\mathcal{O}_q)) - \Var(T^\textnormal{B}(0,\mathcal{O}_q))\right|. \label{eq: end_3}
	\end{align}
	By \eqref{eq: smores_quest_bar}, \eqref{eq: end_3} is $o(q) = o(\log n)$. The term \eqref{eq: end_2} is a special case of \eqref{eq: end_1} (with Bernoulli weights). To bound \eqref{eq: end_1}, we use the inequality
	\[
	|\Var(X_1) - \Var(X_2)| \leq \Var(X_1-X_2) + 2\sqrt{\Var (X_1-X_2) \Var(X_2)}.
	\]
	(To derive this, one sets $X' = X - \E X$ for a random variable $X$ and writes the left side as 
	\[
	| \|X_1'\|_2^2 - \|X_2'\|_2^2| \leq \|X_1'-X_2'\|_2 (\|X_1'-X_2'\|_2 + 2\|X_2'\|_2.)
	\]
	We put $X_1 = T(0,\partial B(n))$ and $X_2 = T(0,\mathcal{O}_q)$ to bound \eqref{eq: end_1}, noting that \eqref{eq: strawberry_cheesecake_quest_bar} implies that $\Var(X_1-X_2) \leq C_{40}$. We obtain
	\[
	C_{40} + 2\sqrt{C_{40} \Var T(0,\mathcal{O}_q)}.
	\]
	Finally, $\Var T(0,\mathcal{O}_q) = \sum_{k=0}^q \E \Delta_k^2$, which, by Lemma~\ref{lem: preliminary_bounds}, is bounded by $C_{44}q$. Therefore the sum of \eqref{eq: end_1}, \eqref{eq: end_2}, and \eqref{eq: end_3} is bounded above by
	\[
	2C_{40} + 4\sqrt{C_{40} C_{44}q} + o(\log n) = o(\log n).
	\]
	This completes the proof of \eqref{eq: var}.


\end{document}